\documentclass[]{interact}

\usepackage[caption=false]{subfig}% Support for small, `sub' figures and tables

\usepackage[font = small, labelfont = bf, labelsep = period]{caption}
\usepackage[ruled,vlined,linesnumbered]{algorithm2e}
\usepackage{booktabs}
\usepackage{tikz}
\usetikzlibrary{positioning,snakes}
\tikzset{
  main node/.style={circle,fill=black,draw,minimum size=2mm,inner sep=0pt},
}

\usepackage[numbers,sort&compress]{natbib}% Citation support using natbib.sty
\bibpunct[, ]{[}{]}{,}{n}{,}{,}% Citation support using natbib.sty
\theoremstyle{plain}% Theorem-like structures provided by amsthm.sty
\newtheorem{theorem}{Theorem}[section]

\newtheorem{proposition}[theorem]{Proposition}

\theoremstyle{definition}

\theoremstyle{remark}

\title{Parallel Interior-Point Solver for Block-Structured Nonlinear Programs on SIMD/GPU Architectures}
\author{
  \name{
  François Pacaud\textsuperscript{a},
  Michel Schanen\textsuperscript{b},
  Sungho Shin\textsuperscript{b},
  Daniel Adrian Maldonado\textsuperscript{b},
  Mihai Anitescu\textsuperscript{b}
}
  \affil{
    \textsuperscript{a} Centre Automatique et Systèmes, Mines Paris - PSL, Paris, France;
    \textsuperscript{b} Mathematics and Computer Science Department, Argonne National Laboratory, Lemont, USA}
}
\date{\today}

\begin{document}
\maketitle

\begin{abstract}
  We investigate how to port the standard interior-point
  method to new exascale architectures for block-structured nonlinear programs
  with state equations.
  Computationally, we decompose the interior-point algorithm
  into two successive operations: the evaluation of the derivatives and
  the solution of the associated Karush-Kuhn-Tucker (KKT) linear system.
  Our method accelerates both operations using two levels of parallelism.
  First, we distribute the computations on multiple processes using
  coarse parallelism.
  Second, each process uses a SIMD/GPU accelerator locally
  to accelerate the operations using fine-grained parallelism.
  The KKT system is reduced by eliminating the inequalities and the state variables from the corresponding equations,
  to a dense matrix encoding the sensitivities of the problem's degrees of freedom,
  drastically minimizing the memory exchange.
  We demonstrate the method's capability on the supercomputer Polaris,
  a testbed for the future exascale Aurora system. Each node is equipped
  with four GPUs, a setup amenable to our two-level approach.
  Our experiments on the stochastic optimal power flow problem
  show that the method can achieve a 50x speed-up compared to the state-of-the-art
  method.
\end{abstract}

\section{Introduction}

Solving complex engineering problems often resorts
to the solution of large-scale block-structured nonlinear programs.
As such, there has been a long interest in designing efficient nonlinear optimization
algorithms, particularly by using parallel computing.
Parallelism can happen at two levels.
At first, \emph{coarse parallelism}
splits the program into large computational chunks, usually dispatched to multiple processors
using a message-passing interface in distributed memory.
In this paradigm, the parallel algorithm is designed to minimize the communication
between the different processes.
In a complementary direction, \emph{fine-grained parallelism} breaks down the program
into small tasks, fast to compute in shared memory. This method requires
a large number of processors to be efficient, and it is usually better
on SIMD architectures with low communication overhead, as provided by Graphical Processing
Units (GPUs).
In the mathematical optimization community, coarse parallelism
has traditionally been used to solve large-scale block-structured optimization
problems, as encountered in dynamic or stochastic nonlinear programs.
On the contrary, fine-grained parallelism has gained attraction only
recently, with the renewed interests
for machine learning applications and stochastic gradient algorithms.
In this work, we combine coarse and fine-grained parallelism
to solve block-structured nonlinear problems on new exascale architectures, where the solution algorithm
is streamlined on different GPUs using CUDA-aware MPI.
% To not make these results specific to exascale architectures, I would clarify
% "solve block-structured nonlinear problems on high-performance computing architectures where the solution algorithm is streamlined on different GPUs using CUDA-aware MPI. We will show experiments on Polaris ..."
% What do you think?

\subsection{Literature review}
In his pioneering work~\cite{schnabel1985parallel,schnabel1995view}, Robert Schnabel identified three practical approaches
to run optimization algorithms in parallel:
(i) parallelize the function evaluations;
(ii) parallelize the linear algebra; and
(iii) parallelize the optimization algorithm itself.

The first attempt to parallelize the evaluations has been
to streamline the computation of the derivatives using finite-differences~\cite{nash1991general}. Soon, it has been noted that parallelizing the forward pass
in automatic differentiation (AD) is also straightforward, provided that we can propagate
the tangents (encoding the first-order sensitivity) in parallel~\cite{hovland1997automatic}.
Unfortunately, doing the same in the reverse pass is not trivial,
as adjoining a mutable code leads to race conditions
(e.g., every read becomes a write operation).
This has led to extensive research on adapting automatic differentiation
to parallel environments~\cite{bucker2001bringing,hovland1998automatic,moses2021reverse}.
Now, most state-of-the-art differentiable tools employ a Domain Specific Language (DSL)
constraining the user to specific differentiable operations.
In particular, this approach has been adopted mainly in machine learning,
leading to the development of fast AD libraries efficiently generating the derivatives
efficiently on hardware accelerators such as GPUs or TPUs~\cite{paszke2019pytorch,jax2018github}.

The parallelization of linear algebra is usually more involved,
as most large-scale optimization methods fall back on the solution of sparse indefinite
Karush-Kuhn-Tucker (KKT) systems~\cite{nocedal2006numerical}.
In the 1980s, preliminary results were obtained by running
iterative methods in parallel, using block-Krylov~\cite{saad1980rates} or block-truncated Newton methods \cite{nash1989block}.
However, block iterative algorithms are quickly limited by the lack of
generic preconditioners for KKT systems.
The 1990s witnessed the emergence of
the interior-point methods (IPM), together with the development of
large-scale sparse direct linear solvers
\cite{duff2004ma57,schenk2004solving}.
In IPM, a significant portion of the time is spent solving a sequence of
(indefinite) KKT systems, hence the method directly benefits from efficient
sparse linear solvers able to run in parallel \cite{amestoy2000multifrontal,duff1999developments}.
In the 2000s, it was shown that, for block-structured optimization problems as we consider here, the layout of the optimization problem
can be exploited further in a Schur complement approach to solve the Newton step
in parallel~\cite{birge1988computing,choi1993exploiting,jessup1994parallel,gondzio2003parallel,zavala2008interior,word2014efficient,petra2014augmented}.
These developments led to the development of mature decomposition-based parallel nonlinear solvers for scenario-based problems
in the 2010s~\cite{gondzio2009exploiting,chiang2014structured,zhu2009exploiting,rodriguez2021scalable,shin2021graph}.

Eventually, running an optimization algorithm fully in parallel generally requires a
subtle combination of (i) and (ii), often devolving to
a software engineering problem. The challenge is to evaluate the
derivatives \emph{and} solve the resulting KKT system each in parallel; all this
while minimizing the communication between the different processes.
This has led to the development of different prototypes for MPI-parallel modelers
\cite{colombo2009structure,watson2012pysp,huchette2014parallel,rodriguez2021scalable},
most of them extending a specific AD backend \cite{fourer1990modeling,bussieck2004general,dunning2017jump}.
Such approaches have been successfully applied
to solve large-scale block-structured nonlinear problems, as encountered
in stochastic programming and dynamic optimization.

% structjump
% parapint

\subsection{Contributions}
In this article, we introduce a new parallel algorithm to solve block-structured nonlinear programs involving state equations
on exascale supercomputers.
Our algorithm uses the parallel interior-point solver MadNLP~\cite{shin2021graph}, using
two layers of parallelism to streamline both the evaluation of the derivatives
and the solution of the KKT system.
This framework targets new exascale supercomputers,
where each node is assigned to multiple GPUs connected with
a unified memory (designed to have fast memory exchange between
the different GPUs).

We demonstrate the capability of the algorithm on scenario-based power flow problems (block-OPF), here formulated as two-stage stochastic nonlinear programs. The scenarios can be stochastic or represent contingencies (which can be interpreted as stochastic outcomes with uniform distribution), as is the case of the very widely used security-constrained AC optimal power flow (SC-ACOPF) problem \cite{capitanescu2011state}. SC-ACOPF is one of the core analyses undertaken in the planning, operational planning, and real-time operation of transmission systems \cite{capitanescu2011state}. SC-ACOPF is run several times a day by many operators in the US and the world. For brevity, we will refer to such problems as stochastic.

The block structure of such problems is given by the different scenarios associated with
the stochastic problem, leading to potential parallelism in both
the evaluation of the derivatives and the solution of the resulting block-angular KKT system.
The parallel solution of the block-OPF problem with a Schur complement approach has been studied
extensively both with PIPS-NLP~\cite{chiang2014structured,schanen2018toward} (multiprocessing)
and with Beltistos~\cite{kardovs2019two,kardovs2022beltistos} (multiprocessing
+ factorization of the dense Schur complement on the GPU). Compared to the state-of-the-art
solver Beltistos, our approach carries out almost all computation on the GPUs including a global
CUDA-aware MPI reduction, from the evaluation of the derivatives to the assembling of the Schur complement.
We test our implementation on the pre-exascale supercomputer Polaris, where each node
is equipped with 4 A100 GPUs, and we solve block-OPF problems with up to 9,251 nodes.
\section{Problem statement}

In systems engineering, it is common to encounter optimization problems with relatively
few degrees of freedom -- "controls". Then, the goal is to appropriately fix the values for the degrees of
freedom, e.g., by minimizing a given operational cost while satisfying
the physical equations of the problem.
In that context, the internal state of the system
is described by a \emph{state} variable $x \in \mathbb{R}^{n_x}$,
whose values depend on the current \emph{controls} $u \in \mathbb{R}^{n_u}$
associated with the problem's degrees of freedom.
If the problem is well-posed, this translates to
the \emph{state equation} $g(x, u) = 0$, where
the function $g$ exhibits the physical structure of the problem
(e.g., a differential equation encoding a dynamics,
or a nonlinear network flow associated with static balance equations).
When the system faces uncertainties, it is often appropriate to choose
a control $u$ feasible under a finite set of conditions (or scenarios).
That is, the control $u$ must satisfy $N$ different state equations
\begin{equation}
  \label{eq:state_equation}
  g_i(x_i, u) = 0 \quad \text{for all} \quad i = 1, \cdots, N ,
\end{equation}
where the state $x_i$ now depend on the current scenario $i$. The variables
$x_i$ can be assimilated into a recourse variable.
The $N$ functions $g_1, \cdots, g_N$ define the block structure
of the problem.

\subsection{Block-structured nonlinear programs}
In addition to satisfying the $N$ state equations~\eqref{eq:state_equation},
we aim at minimizing the average operating costs on the $N$ different
scenarios.
The corresponding problem formulates as a two-stage nonlinear program,
which, in our case, is a nonlinear program with
\emph{partially separable structure} \cite{demiguel2008decomposition}:
\begin{equation}
  \label{eq:originalproblem}
  \min_{\substack{x_1, \cdots, x_N, \\ u}} \; \sum_{i=1}^N f_i(x_i, u)
  \quad \text{s.t.} \quad
  \left\{
  \begin{aligned}
    & x_i \geq 0 \; , \quad  u \geq 0 \\
    & g_i(x_i, u) = 0 \;, \\
    & h_i(x_i, u) \leq 0 \; ,
  \end{aligned}
  \right.
  \quad \forall i = 1, \cdots, N \; ,
\end{equation}
with $f_i: \mathbb{R}^{n_x} \times \mathbb{R}^{n_u} \to \mathbb{R}$,
$g_i:  \mathbb{R}^{n_x} \times \mathbb{R}^{n_u} \to \mathbb{R}^{n_x}$,
$h_i:  \mathbb{R}^{n_x} \times \mathbb{R}^{n_u} \to \mathbb{R}^{m}$ smooth functions
encoding the objective, the state equations, and the operational constraints, respectively.
We note that the number of variables ($N \times n_x + n_u$) and constraints
($N \times (m+n_x)$) are linearly proportional to the number of blocks $N$.

In addition, if we introduce local control variables $u_1, \cdots, u_N$ with
the additional coupling constraint $u_1 = \cdots = u_N = u$, we get
a problem with a separable structure, solvable using the primal
decomposition method; at the expense of increasing the search space \cite{ruszczynski1993interior,demiguel2008decomposition}.

By introducing slack variables $s_1, \cdots, s_N$, we rewrite \eqref{eq:originalproblem}
in standard form:
\begin{equation}
  \label{eq:slackedproblem}
  \min_{\substack{x_1, \cdots, x_N, \\ s_1, \cdots, s_N,\\  u}} \; \sum_{i=1}^N f_i(x_i, u)
  \quad \text{s.t.} \quad
  \left\{
  \begin{aligned}
    & u \geq 0 \; , \quad  x_i \geq 0  \; , \quad s_i \geq 0 \\
    & g_i(x_i, u) = 0 \;, \\
    &  h_i(x_i, u) + s_i = 0 \; ,
  \end{aligned}
  \right.
  \quad \forall i = 1, \cdots, N \; .
\end{equation}

We define $y_i \in \mathbb{R}^{n_x}$ the multipliers (or \emph{adjoints}) associated
to the equality constraints $g_i(x_i, u) = 0$,
$z_i \in \mathbb{R}^{m}$ the multipliers associated
to the operational constraints $h_i(x_i, u) + s_i = 0$,
as well as $\lambda, \kappa_i, \nu_i$ the three multipliers
associated to the respective bound constraints $u \geq 0, x_i \geq 0, s_i \geq 0$.
The Lagrangian associated to \eqref{eq:slackedproblem} is:
\begin{multline}
  \label{eq:lagrangian}
  L(x, u, s; y, z, \lambda, \mu, \nu) := \\ \sum_{i=1}^N \Big[ f_i(x_i, u)
    + y_i^\top g_i(x_i, u) + z_i^\top \big(h_i(x_i, u) + s_i\big)
    - \kappa_i x_i - \nu_i s_i
  \Big] - \lambda u \; ,
\end{multline}
with $x := (x_1, \cdots, x_N)$,
$s := (s_1, \cdots, s_N)$,
$y := (y_1, \cdots, y_N)$,
$z := (z_1, \cdots, z_N)$.
To simplify the notations, we define the \emph{extended}
objective function and the \emph{extended} constraints:
\begin{equation*}
  f(x, u) := \sum_{i=1} f_i(x_i, u) \; , \quad
g(x, u) := \begin{bmatrix}
  g_1(x_1, u) \\
  \vdots \\
  g_N(x_N, u)
\end{bmatrix}
\;, \quad
h(x, u) := \begin{bmatrix}
  h_1(x_1, u) \\
  \vdots \\
  h_N(x_N, u)
\end{bmatrix}
.
\end{equation*}
We assume the functions $f, g, h$ are twice differentiable.
We denote
\begin{equation*}
  \begin{aligned}
    & H = \partial_{(x, u)} h(x, u) \in \mathbb{R}^{Nm \times (Nn_x+n_u)}& \text{\small Jacobian of the inequality cons.} \\
    & G = \partial_{(x, u)} g(x, u) \in \mathbb{R}^{Nn_x \times (Nn_x+n_u)}& \text{\small Jacobian of the equality cons.} \\
      & W = \nabla_{(x, u)}^2 L(x, u, s; \cdot)
      \in \mathbb{R}^{(Nn_x+n_u) \times (Nn_x+n_u)}
    & \text{\small Hessian of Lagrangian.}
  \end{aligned}
\end{equation*}

\subsection{Interior-point method}
The interior-point method (IPM) \cite[Chapter 19]{nocedal2006numerical} is a classical
approach to solve~\eqref{eq:slackedproblem}.

\subsubsection{KKT system}
The Karush-Kuhn-Tucker (KKT) equations associated to~\eqref{eq:slackedproblem} can be expressed as
\begin{subequations}
  \label{eq:kkt:fullspace}
\begin{align}
  & \nabla_x f_i + (G_x^i)^\top y_i + (H_x^i)^\top z_i - \kappa_i = 0  , & \forall i=1,\cdots,N\\
  & \sum_{i=1}^N \Big(\nabla_u f_i + (G_u^i)^\top y_i + (H_u^i)^\top z_i \Big) - \lambda = 0  , & \text{(coupling)}\\
  & z_i - \nu_i = 0  ,& \forall i=1,\cdots,N \\
  & g_i(x_i, u) = 0  , & \forall i=1,\cdots,N\\
  & h_i(x_i, u)  + s_i = 0,  & \forall i=1,\cdots,N\\
  \label{eq:kkt:cond1}
  & X_i \kappa_i = 0, \; (x_i,\kappa_i) \geq 0,& \forall i=1,\cdots,N \\
  \label{eq:kkt:cond2}
  & S_i \nu_i = 0, \; (s_i,\nu_i) \geq 0, & \forall i=1,\cdots,N\\
  & U \lambda= 0, \; (u, \lambda) \geq 0,
  \label{eq:kkt:cond3}
\end{align}
\end{subequations}
where $U = \text{diag}(u)$, $X_i = \text{diag}(x_i)$, $S_i = \text{diag}(s_i)$.

The interior-point method uses a homotopy parameter $\mu > 0$ to
replace the complementarity constraints~\eqref{eq:kkt:cond1}-\eqref{eq:kkt:cond2}-\eqref{eq:kkt:cond3}
by the smooth approximations: $X_i \kappa_i = \mu e_{n_x}, S_i \nu_i = \mu e_{m},
U \lambda = \mu e_{n_u}$ ($e_n$ being the vector of all ones of dimension $n$).
The resulting (smooth) system of nonlinear equations can be solved iteratively
using Newton method, where at each iteration, the descent direction
is updated by solving the following \emph{augmented} linear system:
\begin{equation}
  \label{eq:kkt:augmented}
  \begin{bmatrix}
    W + \Sigma_{p} & 0 & G^\top & H^\top \\
    0 & \Sigma_s & 0 & I \\
    G & 0 & 0 & 0\\
    H & I & 0 & 0
  \end{bmatrix}
  \begin{bmatrix}
    p_{d} \\
    p_{s} \\
    p_{y} \\
    p_{z}
  \end{bmatrix}
  =
  -
  \begin{bmatrix}
    r_{1} \\
    r_{2} \\
    r_{3} \\
    r_{4}
  \end{bmatrix}
\end{equation}
with $r_1 =
\begin{bmatrix}
  \nabla_x f + G_x^\top y + H_x^\top z - \mu X^{-1} e_{n_x} \\
  \nabla_u f + G_u^\top y + H_u^\top z - \mu U^{-1} e_{n_u}
\end{bmatrix}
$,
$r_2= z - \mu S^{-1} e_m$, $r_3 = g(x, u)$, $r_4 = h(x, u) + s$.
The primal descent direction $p_d$ decomposes as
$p_d = (p_{x_1}, \cdots, p_{x_N}, p_u)$.

\subsubsection{Block angular structure}
The linear system~\eqref{eq:kkt:augmented} is sparse and symmetric indefinite,
and can be factorized using the Bunch-Kaufman algorithm.
However, it is often beneficial to exploit its block-angular structure.
Indeed, both the Hessian of the Lagrangian and the Jacobians
have a block-angular structure, given as
\begin{equation*}
  W = \begin{bmatrix}
    W_{x_1 x_1} &  && W_{x_1u}\\
     & \ddots & & \vdots \\
      &  & W_{x_N x_N} & W_{x_N u} \\
    W_{u x_1} & \hdots & W_{u x_N} & W_{uu}
  \end{bmatrix}
  , \quad
  G = \begin{bmatrix}
    G^1_{x_1} & &  &  G^1_{u} \\
            & \ddots && \vdots \\
             & & G^N_{x_N} &  G^N_{u}
  \end{bmatrix}
  \; .
  % \quad
  % H = \begin{bmatrix}
  %   H^1_{x_1} & & &  H^1_{u} \\
  %             & \ddots && \vdots \\
  %            & & H^N_{x_N} &  H^N_{u}
  % \end{bmatrix}
\end{equation*}
By reordering the linear system \eqref{eq:kkt:augmented}, we can
expose the block-angular structure of the KKT system as:
\begin{equation}
  \label{eq:kkt:arrowheadsystem}
  \begin{bmatrix}
    A_1 & & & B_1^\top \\
    & \ddots & & \vdots \\
    &   & A_N & B_N^\top \\
    B_1 & \hdots & B_N & A_0
  \end{bmatrix}
\end{equation}
with
\begin{equation*}
  A_0 = W_{uu}
  , \quad
  A_i = \begin{bmatrix}
    W_{x_ix_i} +\Sigma_{x_i} & 0 & G_{x_i}^\top & H_{x_i}^\top \\
    0 & \Sigma_{s_i} & 0 & I  \\
    G_{x_i} & 0 & 0 & 0 \\
    H_{x_i} & I & 0 & 0
  \end{bmatrix}
  , \quad
  B_i = \begin{bmatrix}
    W_{x_iu} \\
    (G_{u}^i)^\top \\
    (H_{u}^i)^\top
  \end{bmatrix}^\top \; .
\end{equation*}
The block-angular structure~\eqref{eq:kkt:arrowheadsystem} can be
exploited to solve the KKT linear system in parallel using a
Schur complement approach. In that case, the submatrices $A_i$
can be factorized independently to assemble the Schur complement in parallel~\cite{chiang2014structured}.

\subsection{Condensation and reduction}
\label{sec:problem:kkt}
Instead of reordering the augmented KKT system~\eqref{eq:kkt:augmented}
as a block angular matrix~\eqref{eq:kkt:arrowheadsystem},
we propose an alternative approach based on successive condensation and reduction
of the KKT system, following the method introduced in \cite{pacaud2022condensed}.
If the structure is well-defined, we show that we can condense the
KKT system~\eqref{eq:kkt:augmented} to a dense matrix
with size $n_u \times n_u$ in two steps: first, by removing the inequality constraints
in \eqref{eq:kkt:augmented}, then by exploiting the structure of
the equality constraints to reduce the condensed system to a dense matrix.
The condensation and reduction steps are illustrated
in Figure~\ref{fig:problem:kktreduction}.

\begin{figure}[!ht]
  \centering
  \includegraphics[width=15cm]{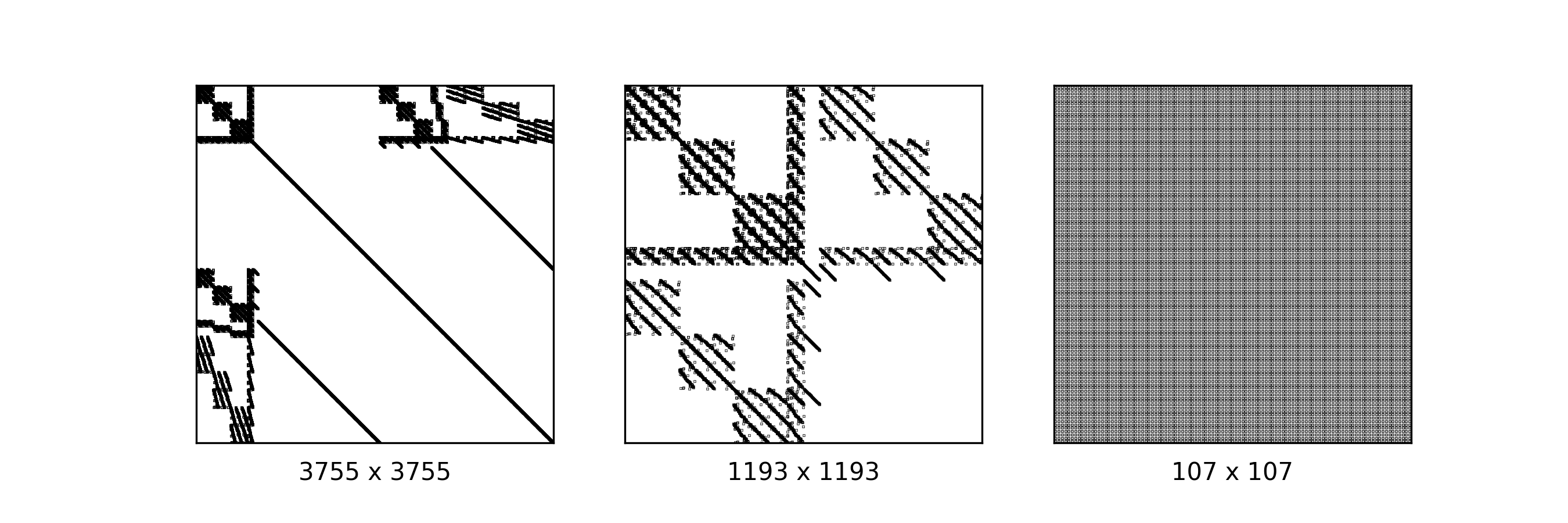}
  \caption{Successive reductions for a block-structured nonlinear problem with $N=3$: Augmented system \eqref{eq:kkt:augmented}, Condensed system \eqref{eq:kkt:condensed}, Reduced system \eqref{eq:kkt:reduced}.
  \label{fig:problem:kktreduction}}
\end{figure}

\subsubsection{Condensation step}
\label{sec:problem:kkt:condensation}
The condensation step allows reducing the size of the KKT system drastically
if the number of inequality constraints is large\footnote{It is equivalent to the normal
equations in linear programming~\cite[Chapter 16, p.412]{nocedal2006numerical}}.

\begin{proposition}[Condensed KKT system]
  \label{prop:condensed}
  The linear system~\eqref{eq:kkt:augmented} is equivalent to
  \begin{equation}
    \label{eq:kkt:condensed}
    \begin{bmatrix}
      K + \Sigma_{p} & G^\top  \\
      G & 0
    \end{bmatrix}
    \begin{bmatrix}
      p_{d} \\
      p_{y}
    \end{bmatrix}
    =
    -
    \begin{bmatrix}
      r_{1} + H^\top(\Sigma_s r_4 - r_2) \\
      r_{3}
    \end{bmatrix}
    \; ,
  \end{equation}
  where $K \in \mathbb{R}^{(Nn_x + n_u) \times (Nn_x + n_u)}$ is the \emph{condensed matrix}
  $K := W + H^\top \Sigma_s H$.
  The descent directions $p_s$ and $p_z$ are recovered as
  \begin{equation}
    \left\{
    \begin{aligned}
      & p_z = \Sigma_s \big[H p_{d} + r_4 \big] - r_2 \; , \\
      & p_s = -\Sigma_s^{-1} \big[r_2 + p_z \big] \; .
    \end{aligned}
    \right.
  \end{equation}
\end{proposition}
\begin{proof}
  See \cite[Theorem 2.2]{pacaud2022condensed}.
\end{proof}
The condensed matrix $K$ inherits the block-angular structure
of the Hessian of the Lagrangian $W$.

\begin{proposition}
  The condensed matrix $K = W + H^\top \Sigma_s H$
  has a block-angular structure, given as
  \begin{equation}
    \label{eq:k_arrowhead}
    K =
    \begin{bmatrix}
      K_{x_1 x_1} &  && K_{x_1u}\\
      & \ddots & & \vdots \\
        &  & K_{x_N x_N} & K_{x_N u} \\
      K_{u x_1} & \hdots & K_{u x_N} & K_{uu}
    \end{bmatrix}
  \end{equation}
  where we have defined the condensed blocks
  $K_{x_i x_i} := W_{x_i x_i} + (H_{x_i}^i)^\top \Sigma_{s_i} H^i_{x_i}$,
  $K_{u x_i} := W_{u x_i} + (H^i_{u})^\top \Sigma_{s_i} H^i_{x_i}$ and
  $K_{uu} := W_{uu} + \sum_{i=1}^N (H^i_{u})^\top \Sigma_{s_i} H^i_{u}$.
\end{proposition}
\begin{proof}
  This is proved by induction.
\end{proof}

\subsubsection{Reduction step}
\label{sec:problem:kkt:reduction}
In addition, we can exploit the structure of the equality constraints
$g_1, \cdots, g_N$ to further reduce the size of the linear system
\eqref{eq:kkt:condensed} down to a dense matrix with size $n_u \times n_u$.
Equation~\eqref{eq:k_arrowhead} exhibits the structure w.r.t.
the state $x$ and the control $u$, we rewrite as such the
condensed KKT system~\eqref{eq:kkt:condensed} as
\begin{equation*}
  \begin{bmatrix}
  K_{x_1 x_1} &  && K_{x_1u} & (G^1_{x_1})^\top & & \\
              & \ddots & & \vdots & & \ddots & \\
              &  & K_{x_N x_N} & K_{x_N u} & & & (G^N_{x_N})^\top \\
  K_{u x_1} & \hdots & K_{u x_N} & K_{uu} & (G^1_{u})^\top & \hdots & (G^1_{u})^\top \\
  G^1_{x_1} & &  & G^1_{u} & & & \\
  & \ddots &  &\vdots  & & & \\
  & & G^N_{x_N} & G^N_{u} & & & \\
  \end{bmatrix}
  \begin{bmatrix}
    p_{x_1} \\
    \vdots \\
    p_{x_N} \\
    p_u \\
    p_y^1 \\
    \vdots \\
    p_y^N
  \end{bmatrix}
  =
  -
  \begin{bmatrix}
    \hat{r}_1^1 \\
    \vdots \\
    \hat{r}_1^N \\
  \hat{r}_2 \\
    \hat{r}_3^1 \\
    \vdots \\
    \hat{r}_3^N
  \end{bmatrix}
  \; ,
\end{equation*}
where we have renamed the right-hand-side in~\eqref{eq:kkt:condensed}
as $\hat{r}$.

\begin{proposition}[Reduction]
  \label{prop:reduction}
  Assume that for all $i=1,\cdots, N$ the Jacobian
  matrices $G_x^i \in \mathbb{R}^{n_x \times n_x}$ are invertible.
  Then the linear system~\eqref{eq:kkt:condensed}
  is equivalent to
  \begin{equation}
    \label{eq:kkt:reduced}
    \hat{K}_{uu} \, p_u = - \hat{r}_2
    + \sum_{i=1}^N \Big[ (G_u^i)^\top (G_x^i)^{-\top} \hat{r}_1^i
    + \big[K_{ux_i} - (G_u^i)^\top (G_x^i)^{-\top} K_{x_ix_i}  \big] (G_x^i)^{-1}\hat{r}_3^i \Big]
  \end{equation}
  with $\hat{K}_{uu} := Z^\top K Z$
  and $Z \in \mathbb{R}^{(n_u + N n_x) \times n_u}$ is the reduction operator defined as
  \begin{equation}
    Z = \begin{bmatrix}
      -(G_x^1)^{-1} G_u^1 \\
      \vdots \\
      -(G_x^N)^{-1} G_u^N \\
      I
    \end{bmatrix}
    \; .
  \end{equation}
  The descent directions $p_x$ and $p_y$
  are recovered as
  \begin{equation}
    \left\{
    \begin{aligned}
      & p_x^i = -(G_x^i)^{-1} \big[\hat{r}_3^i + G_u^i p_u \big] \\
      & p_y^i = -(G_x^i)^{-\top} \big[\hat{r}_1^i + K_{x_i x_i}p_x^i + K_{x_i u}p_u \big] \; .
    \end{aligned}
    \right.
  \end{equation}
\end{proposition}
\begin{proof}
  See \cite[Theorem 2.1]{pacaud2022condensed}.
\end{proof}

The reduction \eqref{eq:kkt:reduced} is equivalent to a Schur complement
approach applied to the condensed KKT system~\eqref{eq:kkt:condensed}.
In Proposition~\eqref{prop:condensed}, we have shown that the condensed
matrix $K$ has a block-angular structure. The associated
condensed KKT system~\eqref{eq:kkt:condensed} is also inheriting a block-angular
structure in the form of~\eqref{eq:kkt:arrowheadsystem},  where
the blocks are given by
\begin{equation}
  \label{eq:kkt:blockangularcondensed}
  A_0 = K_{uu}
  \;, \quad
  A_i = \begin{bmatrix}
    K_{x_i x_i} & (G_{x}^i)^\top \\
    G_{x}^i & 0
\end{bmatrix}
  \;, \quad
  B_i = \begin{bmatrix}
    K_{x_i u} \\
    G_u^i
  \end{bmatrix}^\top
  \; .
\end{equation}
\begin{proposition}
  Assume that for each $i=1, \cdots, N$ the Jacobian $G_x^i$ is invertible.
  Let $S_{uu} = A_0 - \sum_{i=1}^N B_i A_i^{-1} B_i^\top$ be the
  Schur complement associated to the block-angular system
  \eqref{eq:kkt:arrowheadsystem} with the matrices $(A_i, B_i)$
  defined in \eqref{eq:kkt:blockangularcondensed}. Then,
  the Schur complement $S_{uu}$ is equal to the reduced
  matrix $\widehat{K}_{uu}$ defined in \eqref{eq:kkt:reduced}: $S_{uu} = Z^\top K Z$.
\end{proposition}
\begin{proof}
  First, note that if the Jacobian $G_x^i$ is invertible,
  then the block matrix $A_i$ defined in \eqref{eq:kkt:blockangularcondensed}
  is also invertible, with
  \begin{equation}
    \label{eq:proof:invertiblematrix}
    A_i^{-1} =
    \begin{bmatrix}
      0 & (G_x^i)^{-1} \\
      (G_x^i)^{-\top} & - (G_x^i)^{-\top} K_{x_i x_i} (G_x^i)^{-1}
    \end{bmatrix}
    \;.
  \end{equation}
  Using \eqref{eq:kkt:blockangularcondensed}-\eqref{eq:proof:invertiblematrix}, we expand the
  expression of the terms in the sum constituting the Schur complement $S_{uu}$:
\begin{equation*}
  \begin{aligned}
    B_i A_i^{-1} B_i^\top &=
  \begin{bmatrix}
    K_{ux_i} & (G_u^i)^\top
  \end{bmatrix}
  \begin{bmatrix}
      0 & (G_x^i)^{-1} \\
      (G_x^i)^{-\top} & - (G_x^i)^{-\top} K_{x_i x_i} (G_x^i)^{-1}
  \end{bmatrix}
  \begin{bmatrix}
    K_{x_iu} \\ G_u^i
  \end{bmatrix}\;, \\
                          &=
    (G_u^i)^\top (G_x^i)^{-\top} K_{x_iu} + K_{ux_i} (G_x^i)^{-1} (G_u^i) - (G_u^i)^\top (G_x^i)^{-\top} K_{x_i x_i} (G_x^i)^{-1} G_u^i \;.
  \end{aligned}
\end{equation*}
Hence, the Schur complement $S_{uu} = A_0 - \sum_{i=1}^N B_i A_i^{-1} B_i^\top$
expands as
\begin{equation*}
  \begin{aligned}
    S_{uu} &= K_{uu} - \sum_{i=1}^N \big[ (G_u^i)^\top (G_x^i)^{-\top} K_{x_iu} + K_{ux_i} (G_x^i)^{-1} (G_u^i) - (G_u^i)^\top (G_x^i)^{-\top} K_{x_i x_i} (G_x^i)^{-1} G_u^i \big] \\
      &= Z^\top K Z \;.
  \end{aligned}
\end{equation*}
We recover the expression of the reduced matrix $\widehat{K}_{uu}$ in Proposition~\ref{prop:reduction}.
\end{proof}

\subsection{Discussion}
Hence, we can interpret the reduction step as a Schur complement
approach. Forming the Schur complement has always been the bottleneck
when solving distributed block angular problems in parallel~\cite{lubin2011scalable,chiang2014structured}. Its reduction operation
involves large memory transfers between the processes, with the number of transfers being on the order of $\mathcal{O}(log(p))$, where $p$ is the number of processes.
Due to the quasi-shared memory architecture on GPUs, the reduction can be implemented efficiently \cite{pacaud2022condensed}.
In the next section, we propose to extend \cite{pacaud2022condensed}
to assemble the reduced matrix $\widehat{K}_{uu}$ using two levels of
parallelism, using both MPI and CUDA, thus reducing the reliance on distributed memory.

\section{Parallel implementation}
\label{sec:algo}

In the previous section, we have detailed the structure of
block-angular nonlinear programs and presented the condensation and
reduction steps for the KKT system. The loose coupling between the
blocks is favorable
for parallelizing the evaluation of the derivatives
and the solution of the block-angular KKT system.
\emph{Globally}, we can distribute the computation
on different processes using MPI (coarse parallelism).
\emph{Locally}, we can further streamline
the computation using GPU accelerators (fine-grained parallelism).
This paradigm, with its two levels of parallelism, is directly in line with what is currently offered
by the new exascale architectures, where each node has 4 to 8 GPUs,
all sharing a unified memory for fast communication.
We present in \S\ref{sec:algo:ad} how we streamline the evaluation of the model
using automatic differentiation, and in \S\ref{sec:algo:kkt} how we parallelize
the solution of the KKT system.

\subsection{Parallel automatic-differentiation}
\label{sec:algo:ad}
First, we present how to evaluate the model in parallel using automatic differentiation \cite{griewank2008evaluating}.
We illustrate the procedure in Figure~\ref{fig:algo:parallel_ad}.
The goal of the algorithm is to streamline the evaluation of the $N$ scenarios on
$N/M$ GPUs, $M$ being the number of scenarios evaluated locally on each GPU
(we suppose here that $N$ is a multiple of $M$).

\begin{figure}[!ht]
  \centering
  \begin{tikzpicture}
    \node[inner sep=2pt] (root) at (0,1) {{\tt root}};

    \node[main node] (gpu1) at (-3,-0.3) {};
    \node[main node] (gpu2) at (-1,-0.3) {};
    \node[main node] (gpu3) at (1,-0.3) {};
    \node[main node] (gpu4) at (3,-0.3) {};

    \node (gpu5) at (-3, -1) {\includegraphics[width=1cm]{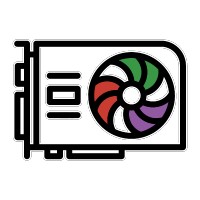}};
    \node (gpu6) at (-1, -1) {\includegraphics[width=1cm]{figures/gpu.png}};
    \node (gpu7) at (1, -1) {\includegraphics[width=1cm]{figures/gpu.png}};
    \node (gpu8) at (3, -1) {\includegraphics[width=1cm]{figures/gpu.png}};

    \draw (root) -- node [midway, above, sloped] {\tiny $g_1,\cdots,g_4$}(gpu1) ;
    \draw (root) -- (gpu2) ;
    \draw (root) -- (gpu3) ;
    \draw (root) -- node [midway, above, sloped] {\tiny $g_{13},\cdots,g_{16}$} (gpu4) ;

    \node (n1) at (-3.6,-1.5) {};
    \node (n2) at (-3.6,-3.0) {\tiny $g_1$};
    \draw[snake=coil,segment length=5pt] (n1) -- (n2) ;

    \node (n3) at (-3.2,-1.5) {};
    \node (n4) at (-3.2,-3.0) {};
    \draw[snake=coil,segment length=5pt] (n3) -- (n4) ;

    \node (n5) at (-2.8,-1.5) {};
    \node (n6) at (-2.8,-3.0) {};
    \draw[snake=coil,segment length=5pt] (n5) -- (n6) ;

    \node (n7) at (-2.4,-1.5) {};
    \node (n8) at (-2.4,-3.0) {\tiny $g_4$};
    \draw[snake=coil,segment length=5pt] (n7) -- (n8) ;

    \node (n9) at (-3.,-3.0) {\tiny $\cdots$};

    \node (p1) at (2.4,-1.5) {};
    \node (p2) at (2.4,-3.0) {\tiny $g_{13}$};
    \draw[snake=coil,segment length=5pt] (p1) -- (p2) ;

    \node (p3) at (2.8,-1.5) {};
    \node (p4) at (2.8,-3.0) {};
    \draw[snake=coil,segment length=5pt] (p3) -- (p4) ;

    \node (p5) at (3.2,-1.5) {};
    \node (p6) at (3.2,-3.0) {};
    \draw[snake=coil,segment length=5pt] (p5) -- (p6) ;

    \node (p7) at (3.6,-1.5) {};
    \node (p8) at (3.6,-3.0) {\tiny $g_{16}$};
    \draw[snake=coil,segment length=5pt] (p7) -- (p8) ;

    \node (p9) at (3.2,-3.0) {\tiny $\cdots$};

    \node (pt1) at (-1,-2.0) {$\cdots$};
    \node (pt2) at (1,-2.0) {$\cdots$};
  \end{tikzpicture}
  \caption{Parallel evaluation of the derivatives for $g_1, \cdots g_N$ on 4 GPUs:
  we have a total of $N=16$ scenarios, each GPU evaluating $M = 16/4 = 4$
  scenarios locally. \label{fig:algo:parallel_ad}}
\end{figure}

% challenge: memory

\subsubsection{Local parallelism}
\label{sec:algo:ad:local}
The first level of parallelism streamlines
the evaluation of the model on SIMD/GPU devices.
We have designed our implementation to run entirely on the GPU device,
to avoid any data transfer between the host and the device.

\paragraph{Block evaluation}
We suppose that the nonlinear functions $(f_i, g_i, h_i)$
share the same structure, its expressions yielding the same Abstract Syntax Tree (AST)
for all $i=1,\cdots, M$.
We illustrate the block evaluation on a simple abstract tree, but the reasoning
extends to more complicated structures. We suppose that for all $i$, the
functions $f_i, g_i, h_i$ depend linearly on a nonlinear basis matrix
$\psi: \mathbb{R}^{n_x} \times \mathbb{R}^{n_u} \to \mathbb{R}^{n_b}$:
that is, there exists three \emph{sparse} matrices $L_f, L_g, L_h$ such that
\begin{equation}
  \label{eq:AST}
  f_i(x_i, u) = L_f \psi(x_i, u) \; , \quad
  g_i(x_i, u) = L_g \psi(x_i, u) \; , \quad
  h_i(x_i, u) = L_h \psi(x_i, u) \; .
\end{equation}
Suppose we aim to evaluate the $M$ functions $g_1, \cdots, g_M$ in batch
for the states $x_1, \cdots, x_M$.
The structure~\eqref{eq:AST} is directly amenable for SIMD evaluation.
We denote by $X_M = (x_1, \cdots, x_M) \in \mathbb{R}^{n_x \times M}$ the dense
matrix obtained by concatenating the $M$ states together. By using a proper
GPU kernel or a parallel modeler, we can evaluate the basis in a SIMD fashion
and build the matrix $\Psi(X_M, u) := \big(\psi(x_1, u), \cdots, \psi(x_M, u)\big) \in \mathbb{R}^{n_b \times M}$.
Then, evaluating the functions $g_1, \cdots, g_M$ simultaneously translates to
the evaluation of one {\tt SpMM} product:
\begin{equation}
  \label{eq:eval:block}
   \big(g_1(x_1, u), \cdots, g_M(x_M, u)\big) = L_g \Psi(X_M, u)
   \in \mathbb{R}^{n_x \times M} \; .
\end{equation}
The total memory required in the two successive operations is
$\mathcal{O}((n_x + n_b) \times M)$, and depends linearly on the number of
blocks $M$.
We note the {\tt SpMM} operations are generally implemented efficiently in the
vendor library ({\tt cusparse} for CUDA, {\tt rocSPARSE} for AMDGPU).

\paragraph{First-order derivatives}
Suppose that for a given $i$ we have a differentiable implementation $\tt{gb}_i: \mathbb{R}^{n_d} \to \mathbb{R}^{n_x}$ associated to the function $g_i$.
We aim to evaluate the Jacobian-matrix
products $(\nabla g_i) D$ for $p$ tangents encoded in a matrix $D \in \mathbb{R}^{n_d \times p}$
using forward-mode AD and operator overloading. This operation translates
to propagating \emph{forward} a vector of dual numbers.
Denoting by $\underline{d} \in \mathbb{D}_p^{n_d}$ the dual number
encoding the $p$ tangents stored in $D$, evaluating $(\nabla g_i) D$
simply amounts to call $\tt{gb}_i(\underline{d})$ and extract
the results in the dual numbers returned as a result.
As $G^i$ is sparse, we can apply the technique of Jacobian coloring \cite{griewank2008evaluating} to
compress the independent columns of the sparse matrix $G^i$ and reduces
the number of required seeding tangents $p$ needed to evaluate the full
Jacobian.

Suppose now we want to evaluate the sparse
Jacobians $G^1, \cdots, G^M$ in batch. As the functions $g_i$ are based on the
same AST, their respective Jacobians $G^1, \cdots, G^M$ are sharing the same sparsity pattern.
By seeding a matrix of dual numbers $\underline{D}_M
= (\underline{d}_1, \cdots, \underline{d}_M) \in \mathbb{D}_p^{n_d \times M}$,
we can use the same operation as \eqref{eq:eval:block} to streamline
the evaluation of the $M$ Jacobian-vector products using
the SIMD kernel $\Psi(\cdot)$ and {\tt SpMM} operations:
\begin{equation}
  \label{eq:eval:blockjac}
  \big(\text{\tt gb}_1(\underline{d}_1), \cdots, \text{\tt gb}_M(\underline{d}_M)\big) := L_g \Psi(\underline{D}_M)
   \in \mathbb{D}_p^{n_x \times M} \; .
\end{equation}
Once the results are evaluated, it remains
to uncompress the dual outputs to build the $M$ sparse Jacobians $G^1, \cdots, G^M$.
Hence, we can streamline the evaluation of the Jacobian along with
the number of tangents $p$ and the number of blocks $M$.
This comes at the expense of increasing memory usage to
$\mathcal{O}((n_x + n_b + n_d) \times M \times p)$ (to store the dual
matrices associated to the input, the intermediate basis $\Psi$ and the output).

\paragraph{Second-order derivatives}
The evaluation of the second-order derivatives follows
the same procedure, using \emph{forward-over-reverse} AD.
For each $i$, we suppose available an adjoint function
$\text{\tt adj\_gb}_i: \mathbb{R}^{n_d} \times \mathbb{R}^{n_x} \to \mathbb{R}^{n_d}$
which for any primal $x \in \mathbb{R}^{n_d}$ and adjoint $y \in \mathbb{R}^{n_x}$
evaluates the Jacobian-transpose vector product $(G^i(x))^\top y$ (reverse-mode).
Using forward-mode AD on top of $\text{\tt adj\_gb}_i$, we can compute the second-order
derivatives $y^\top\nabla^2 g_i(x) V$ for $p$ directions $V$ by calling
$\text{\tt adj\_gb}_i(\underline{x}, \underline{y})$.
Using Hessian coloring, we can compress the independent columns
of the sparse matrix $y^\top\nabla^2 g(x)$ and reduce the number
of seeding tangents $p$ required to evaluate the full Hessian.
We note that in general obtaining an adjoint $\text{\tt adj\_gb}_i$
running in parallel is nontrivial due to potential race conditions incurred by the control flow reversal of the original code.

Computing the Hessian $y^\top\nabla^2 g_i(x)$ in parallel
for $i=1, \cdots M$ amounts to defining two matrices
of dual numbers
$\underline{X}_M = (\underline{X}_1, \cdots, \underline{X}_M) \in \mathbb{D}_p^{n_d \times M}$,
$\underline{Y}_M = (\underline{y}_1, \cdots, \underline{y}_M) \in \mathbb{D}_p^{n_x \times M}$ and evaluate
$\nabla \Psi(\underline{X}_M)^\top L_g^\top \underline{Y}_M$.
The dual outputs are uncompressed  to build the $M$ sparse Hessians
(as the sparsity pattern of the Hessians is different than those of
the Jacobians, the matrix $\underline{X}_M$ employed here is different than
the one used in \eqref{eq:eval:blockjac}).
The total memory required to store the duals
is $\mathcal{O}((2n_x+n_d + n_b) \times M \times p)$. For more details, we refer
to the vector forward mode as described in~\cite{griewank2008evaluating}.

\subsubsection{Global parallelism}
\label{sec:algo:ad:global}
Now, if we have several GPUs at our disposal, we can
push the parallelism further by distributing the evaluations
using multiprocessing and a Message Passing Interface (MPI) library.
Coming back at our original problem~\eqref{eq:originalproblem},
we illustrate in Figure~\ref{fig:algo:parallel_ad} how to dispatch the evaluation of the $N$ nonlinear constraints
$g_1, \cdots, g_N$ (the same reasoning applies to the objectives
$f_1, \cdots, f_N$ and the inequality constraints $h_1, \cdots, h_N$).
We use the streamlined implementation described in the previous
subsection to evaluate the constraints in a batch of size $M$:
the first GPU evaluates the constraints $g_1, \cdots, g_M$, the second GPU
evaluates $g_{M+1}, \cdots, g_{2M}$, and so on. In total, the evaluation
of the $N$ constraints requires $N/M$ GPUs (if $M=1$, each GPU evaluate
one constraint; if $M=N$, we use only one GPU evaluating all the constraints).

The implementation has been designed to minimize the communication between
the different processes: each batch $g_1, \cdots, g_M$
stores the data it needs \emph{locally}, the only data exchange with the other processes being
the vector of input and the vector of output.
In addition, we will see in the next section
we do not have to transfer the first- and second-order information if
a parallel linear solver is being used.

\subsection{Parallel KKT solver}
\label{sec:algo:kkt}
By exploiting the block-angular structure of the KKT system,
we can solve the Newton step in parallel using a Schur complement approach.
The challenge lies in the computation of the Schur complement
matrix $S = A_0 - \sum_{i=1}^N B_i A_i^{-1} B_i^\top$.
Each product $B_i A_i^{-1} B_i^\top$ requires the factorization of the matrix $A_i$ and the
solution of a linear system with multiple (sparse) right-hand-side
$A_i^{-1} B_i$.
State-of-the-art methods are evaluating the Schur complement
using an \emph{incomplete augmented factorization} applied on
the auxiliary matrix $\begin{bmatrix} A_i & B_i^\top \\ B_i & 0 \end{bmatrix}$,
as currently implemented in the Pardiso linear solver~\cite{petra2014augmented}.
Here, we use an alternative approach building
on the reduced KKT system~\S\ref{sec:problem:kkt:reduction}
(equivalent to the Schur complement approach).
As the reduction can be streamlined on GPU accelerators~\cite{pacaud2022condensed},
this approach can assemble the Schur complement in parallel using CUDA-aware MPI.
We illustrate the parallel computation of the Schur complement
in Figure~\ref{fig:algo:parallel_schur}.

\begin{figure}[!ht]
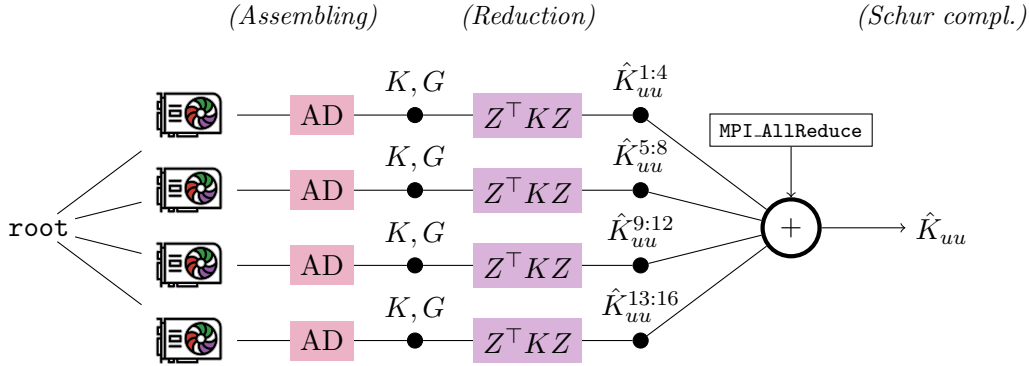

  \centering
  \begin{tikzpicture}
    \node[inner sep=2pt] (root) at (0,0) {{\tt root}};

    \node (gpu1) at (2, 1.5) {\includegraphics[width=1cm]{figures/gpu.png}};
    \node (gpu2) at (2, 0.5) {\includegraphics[width=1cm]{figures/gpu.png}};
    \node (gpu3) at (2, -0.5) {\includegraphics[width=1cm]{figures/gpu.png}};
    \node (gpu4) at (2, -1.5) {\includegraphics[width=1cm]{figures/gpu.png}};

    \draw (root) -- (gpu1) ;
    \draw (root) -- (gpu2) ;
    \draw (root) -- (gpu3) ;
    \draw (root) -- (gpu4) ;

    \node (red) at (3.5, 2.8) {\small\it (Assembling)};

    \node[main node, label=above:{$K,G$}] (ad1) at (5, 1.5) {};
    \node[main node, label=above:{$K,G$}] (ad2) at (5, 0.5) {};
    \node[main node, label=above:{$K,G$}] (ad3) at (5, -0.5) {};
    \node[main node, label=above:{$K,G$}] (ad4) at (5, -1.5) {};

    \draw (gpu1) -- node [midway, fill=purple!30, rectangle] {AD} (ad1) ;
    \draw (gpu2) -- node [midway, fill=purple!30, rectangle] {AD} (ad2) ;
    \draw (gpu3) -- node [midway, fill=purple!30, rectangle] {AD} (ad3) ;
    \draw (gpu4) -- node [midway, fill=purple!30, rectangle] {AD} (ad4) ;

    \node[main node, label=above:{$\hat{K}_{uu}^{1:4}$}] (red1) at (8, 1.5) {};
    \node[main node, label=above:{$\hat{K}_{uu}^{5:8}$}] (red2) at (8, 0.5) {};
    \node[main node, label=above:{$\hat{K}_{uu}^{9:12}$}] (red3) at (8, -0.5) {};
    \node[main node, label=above:{$\hat{K}_{uu}^{13:16}$}] (red4) at (8, -1.5) {};

    \node (red) at (6.5, 2.8) {\small\it (Reduction)};
    \draw (ad1) -- node [midway, fill=violet!30, rectangle] {$Z^\top K Z$} (red1) ;
    \draw (ad2) -- node [midway, fill=violet!30, rectangle] {$Z^\top K Z$} (red2) ;
    \draw (ad3) -- node [midway, fill=violet!30, rectangle] {$Z^\top K Z$} (red3) ;
    \draw (ad4) -- node [midway, fill=violet!30, rectangle] {$Z^\top K Z$} (red4) ;

    \node[draw, rectangle] (schurlab) at (10, 1.3) {\footnotesize\tt MPI\_AllReduce};
    \node[circle, draw, ultra thick] (schur) at (10, 0) {+};
    \draw[->] (schurlab) -- (schur) ;
    \draw (red1) -- (schur);
    \draw (red2) -- (schur);
    \draw (red3) -- (schur);
    \draw (red4) -- (schur);

    \node (finallab) at (12, 2.8) {\small\it (Schur compl.)};
    \node (final) at (12, 0) {$\hat{K}_{uu}$} ;
    \draw[->] (schur) -- (final) ;

  \end{tikzpicture}
  \caption{Parallel computation of the Schur complement.
  \label{fig:algo:parallel_schur}}
\end{figure}

\paragraph*{Assembling the sparse matrices.}
Using the procedure introduced in \S\ref{sec:algo:ad:local},
we evaluate \emph{locally} the Jacobians $G^1, \cdots, G^M$
the Jacobians $H^1, \cdots, H^M$ and the Hessians $W^1, \cdots, W^M$.
Using dedicated kernels,
we uncompress the results in the block-angular sparse Jacobians
\begin{equation*}
  G_x^{1:M} = \begin{bmatrix}
    G_{x_1}^1 & & \\
          & \ddots & \\
          & & G_{x_M}^M
  \end{bmatrix}
  , ~
  G_u^{1:M} = \begin{bmatrix}
    G_u^1 \\
    \vdots \\
    G_u^M
  \end{bmatrix}
  , ~
  H^{1:M} = \begin{bmatrix}
    H_{x_1}^1 & & & H_u^1 \\
              & \ddots & & \vdots \\
              & & H_{x_M}^M & H_u^M
  \end{bmatrix}
  \; ,
\end{equation*}
and sparse Hessian
\begin{equation*}
  W^{1:M} = \begin{bmatrix}
    W_{x_1x_1} & & & W_{x_1 u} \\
              & \ddots & & \vdots \\
              & & W_{x_Mx_M} & W_{x_M u} \\
    W_{ux_1} & \hdots & W_{ux_M} & W_{uu}
  \end{bmatrix}
  \; .
\end{equation*}
Once the sparse matrices are obtained, we recover the
condensed matrix $K^{1:M} = W^{1:M} + (H^{1:M})^\top \Sigma (H^{1:M})$ (Proposition \ref{prop:condensed})
using one {\tt SpGEMM} operation and we factorize the
matrix $G_x^{1:M}$ using a sparse LU factorization (potentially running in
batch as the matrices $G_{x_1}^1, \cdots, G_{x_M}^M$ are sharing
the same sparsity pattern).
Once the matrix $G_x^{1:M}$ is factorized as $PG_x^{1:M}Q = LU$ ($P$, $Q$ being
two permutation matrices), computing $(G_x^{1:M})^{-1}b$ translates to
two backsolves ({\tt SpSV}) and two matrix-vector multiplications ({\tt SpMV}),
as $(G_x^{1:M})^{-1}b = Q U^{-1}L^{-1} P b$.

\paragraph*{Local reduction.}
Once the sparse matrices are built, we evaluate locally the reduced
matrix $\widehat{K}_{uu}^{1:M}$ on the GPU, using $\text{div}(n_u, n_{batch}) + 1$
matrix-matrix product $\widehat{K}_{uu}^{1:M} V$ (with $V \in \mathbb{R}^{n_u \times n_{batch}}$ a dense matrix encoding $n_{batch}$ vectors of the Cartesian basis
of $\mathbb{R}^{n_u}$).
The evaluation of one batched matrix-matrix product $\widehat{K}_{uu}^{1:M} V
= (Z^\top K^{1:M} Z) V$ proceeds in three steps
\begin{enumerate}
  \item Solve $T_x = - (G_x^{1:M})^{-1} (G_u^{1:M} V)$.
  \item Evaluate $\begin{bmatrix}
      L_x \\ L_u
  \end{bmatrix} := \begin{bmatrix}
    K_{x x}^{1:M} & K_{xu}^{1:M} \\
    K_{ux}^{1:M} & K_{uu}^{1:M}
  \end{bmatrix}
  \begin{bmatrix}
    T_x \\ V
  \end{bmatrix}$.
\item Set $\widehat{K}_{uu}^{1:M}V = L_u - G_u^{1:M} (G_x^{1:M})^{-\top} L_x$.
\end{enumerate}
In total, we need 2 {\tt SpSM} and 3 {\tt SpMM} operations
in the first step, 1 {\tt SpMM} in the second step, and 2 {\tt SpSM}
and 3 {\tt SpMM} operations in the third step, giving a total
of 4 {\tt SpSM} and 7 {\tt SpMM} operations.
More than the computation, the reduction is limited by the memory,
as we have to store the three buffers $L_x, T_x, T_u$ with a total
size of $(2M \times n_x + n_u) \times n_{batch}$.
If $n_x$ is too large, it is in our interest to reduce $M$ (by using more GPUs)
or to reduce $n_{batch}$ (at the expense of computing more matrix-matrix
product $\widehat{K}_{uu}^{1:M} V$).

\paragraph*{Global reduction.}
Once we obtain the locally reduced matrices $\widehat{K}_{uu}^{nM+1:(n+1)M}$ for $n=0,\cdots, N/M-1$, we can
assemble the global reduced matrix $\widehat{K}_{uu}= \sum_{n=0}^{N/M-1} \widehat{K}_{uu}^{nM+1:(n+1)M}$
using one \emph{all reduce} ({\tt MPI\_Allreduce}) operation.
The size of the reduced matrix $\widehat{K}_{uu}$ is $n_u \times
n_u$, hence limiting the memory transfer required in the algorithm.

\subsection{Discussion}
We have presented a practical way to assemble the Schur complement
on multi-GPU architectures. The parallelism occurs both at the
\emph{local} level (SIMD evaluations on the GPUs) and at the \emph{global} level
(distributed computation with MPI).
The algorithm has the advantage of assembling the sparse Jacobians and Hessians
only locally, as the reduction occurs before proceeding to the memory transfer
with {\tt MPI\_Allreduce}. The reduced matrix has a dimension $n_u \times n_u$,
which compresses the memory transfer significantly if the number of degrees
of freedom $n_u$ is small. However, this comes at the expense of storing
a vector of dual numbers (whose memory is linearly proportional to the number
of blocks $M$ evaluated locally and the number of tangents $p$ being employed to evaluate
the sparse derivatives) and additional buffers in the reduction algorithm.
In the next section, we will test an implementation of the algorithm
on CUDA GPUs, and show that the algorithm is practical.

\section{Numerical results}
We demonstrate the capabilities of the algorithm
we introduced in Section \S\ref{sec:algo} on the
supercomputer Polaris, using CUDA-aware MPI to dispatch the solution
on multiple GPUs.
We present in \S\ref{sec:results:settings} the stochastic optimal
power flow problem, and give in \S\ref{sec:results:assessments}
detailed assessments of the algorithms we have introduced
earlier in \S\ref{sec:algo}. Eventually, we present in \S\ref{sec:results:blockopf}
a benchmark comparing our parallel solution algorithm with a state-of-the-art
solution method running on the CPU.

\subsection{Settings}
\label{sec:results:settings}

\subsubsection{Case study: the block-structured optimal power flow}
The stochastic optimal power flow problem aims at finding
an optimal dispatch for the generators $u$. The solution $u$
should minimize the operational costs while satisfying the
physical constraints (power flow equations $g(x, u) = 0$, here playing the
role of the state equations) and operational
constraints (line flow constraints $h(x, u) \leq 0$) on a given set of scenarios.
Each scenario is assigned given load parameters (energy demands) and potential
contingencies (line tripping). The values of the state $x$ depend on
the local scenario we are in, the state $x$ being the \emph{recourse} variable in our case.
As such, the problem has a partially separable structure as introduced
in Problem~\eqref{eq:originalproblem}, the control $u$ being shared across all scenarios.
We refer to \cite{capitanescu2011state} for the original presentation
of the stochastic optimal power flow problem
and to \cite{lubin2011scalable,chiang2014structured,kardovs2019two,kardovs2020structure}
for practical algorithms solving the stochastic optimal power flow problem
(some also focus on the multistage setting, which is not covered in this article).
For our benchmark, we look at reference instances provided by
MATPOWER~\cite{zimmerman2010matpower}, whose characteristics are detailed in Table~\ref{tab:results:instances}.
We recall that in our case, the size of the Schur complement matrix $\hat{K}_{uu}$
is given by the number of controls $n_u$.

\begin{table}[!ht]
  \centering
  \begin{tabular}{lrrrrr}
    \toprule
    Name & \#bus & \#lines & \#gen & $n_x$ & $n_u$ \\
    \midrule
    case118 &  118 & 186 & 54 & 181 & 107 \\
    case1354pegase & 1,354 & 1,991 & 260 & 2,447 & 519 \\
    case2869pegase & 2,869 & 4,582 & 510 & 5,227 & 1,019 \\
    case9241pegase & 9,241 & 16,049 & 1,445 & 17,036 & 2,889 \\
    \bottomrule
  \end{tabular}
  \caption{MATPOWER instances used in the benchmark. \label{tab:results:instances}}
\end{table}

\subsubsection{Implementation}
The algorithm has been implemented entirely in Julia 1.8.
The Schur complement approach has been developed as an extension
of the nonlinear optimization solver MadNLP~\cite{shin2021graph},
using CUDA-aware MPI as provided in \cite{byrne2021mpi}.
We have used the package ExaPF as a nonlinear modeler for the optimal power flow problem.
All the results presented here have been generated on the supercomputer
Polaris equipped with a total of 560 nodes, each node having with 1 CPU and 4 A100 GPUs.

\subsection{Assessment of the parallel implementation}
\label{sec:results:assessments}

\subsubsection{Assessing the performance of the parallel automatic differentation}
We first assess the performance of the parallel automatic differentiation
we introduced in \S\ref{sec:algo:ad} in a multi-GPU setting.
We compare the performance we obtain with a CPU implementation.
We use {\tt case1354pegase} as a representative instance, and
display the time spent in the automatic differentiation as we increase
the total number of scenarios $N$.
The results are displayed in Figure~\ref{fig:results:ad}.

We observe that the computation time depends linearly on the number of scenarios,
as expected. For $N=8$, it is not worthwhile dispatching the evaluation on multiple GPUs
as the problem is small enough to be evaluated on a single GPU. For $N=512$,
the evaluation time is 12.3s on the CPU, compared to
0.50, 0.41, 0.31, and 0.28s using 1, 2, 4 and 8 GPUs, respectively.
Hence, we get a 40x speed-up when evaluating the derivatives in a multi-GPU
setting, and it is not worthwhile to use more than 4 GPUs (one node).

\begin{figure}[!ht]
  \centering
  \includegraphics[width=12cm]{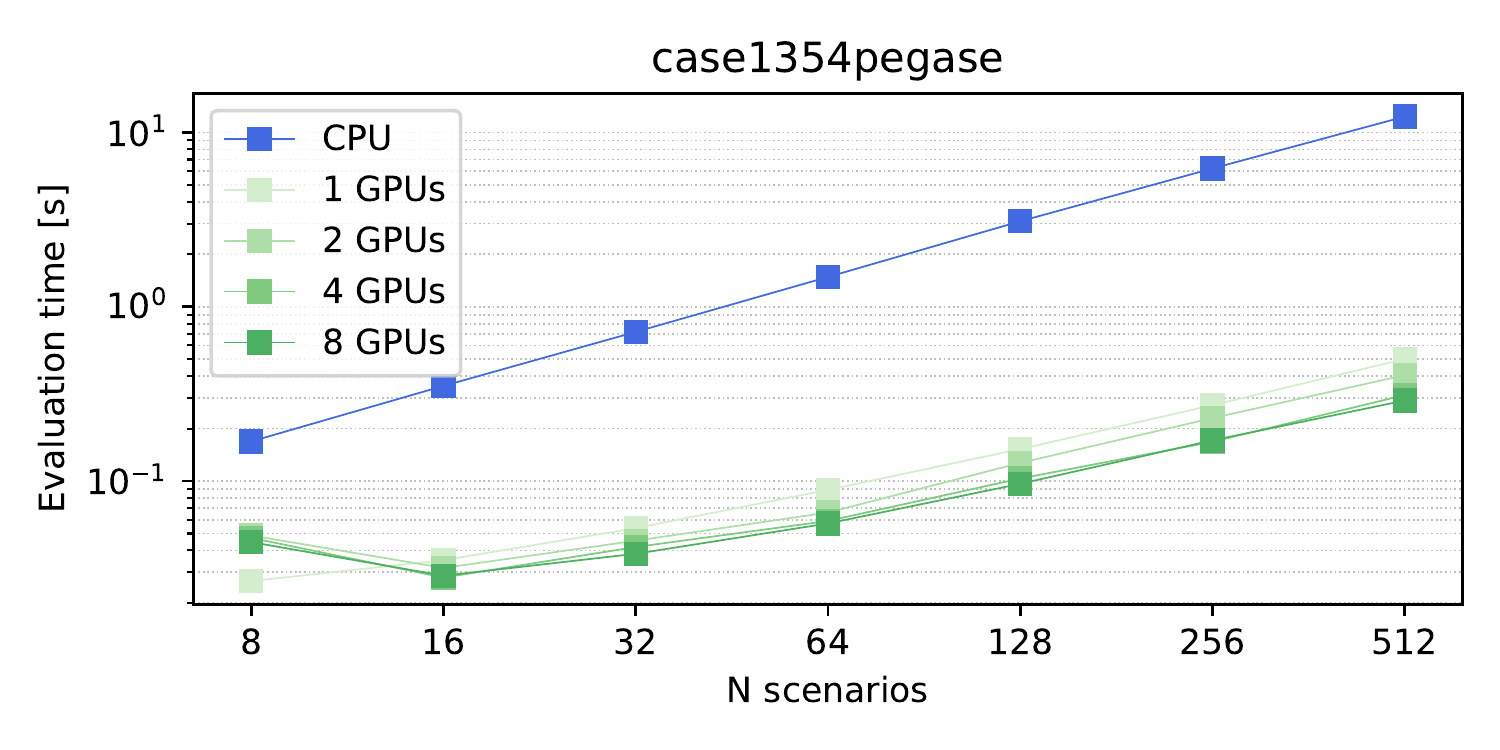}
  \caption{Time spent to evaluate the model and its derivatives
    with automatic differentiation. }
  \label{fig:results:ad}
\end{figure}

\subsubsection{Assessing the performance of the parallel KKT solver}
\label{sec:results:assessments:kkt}
We proceed to the same performance analysis to assess the performance of the
parallel KKT solver detailed in \S\ref{sec:algo:kkt}.
We compare the time required to evaluate the full solution of the KKT system afresh
(including reduction time, factorization time and backsolve time)
on {\tt case1354pegase} as we increase the number of scenarios $N$.
As a reference, we give the time taken by the sparse linear solvers
HSL MA27 (single-threaded) and HSL MA57 (multi-threaded).
The results are displayed in Figure~\ref{fig:results:kkt}.

On the left, we display the evolution of the time spent in the linear solver
as we increase the number of scenarios.
For $N=512$, we observe that we get a linear speed-up as we increase the
number of GPUs: using 8 GPUs, the parallel KKT solver is 40x faster
than using HSL MA27 on the CPU.
Interestingly, we observe that HSL MA57 is not faster than HSL MA27, despite
being multithreaded. This is consistent with the observation
made in \cite{tasseff2019exploring}, and illustrates the difficulty of parallelizing effectively the sparse
LDL factorization (Bunch-Kaufman).
On the right, we display a performance profile detailing
the time spent in MA27 and the parallel KKT solver
on {\tt case1354pegase} with $N=512$ scenarios.
We observe that most of the time in HSL MA27 is spent on factorizing
the sparse augmented KKT system~\eqref{eq:kkt:augmented}.
On the other side, the factorization of the dense reduced matrix $\hat{K}_{uu}$
is trivial using LAPACK on the GPU; the bottleneck in the parallel KKT solver
is the reduction algorithm itself. Fortunately, the reduction algorithm can
run in parallel: we get a linear speed-up as we increase the number of GPUs
used in the reduction algorithm.

\begin{figure}[!ht]
  \centering
  \includegraphics[width=\textwidth]{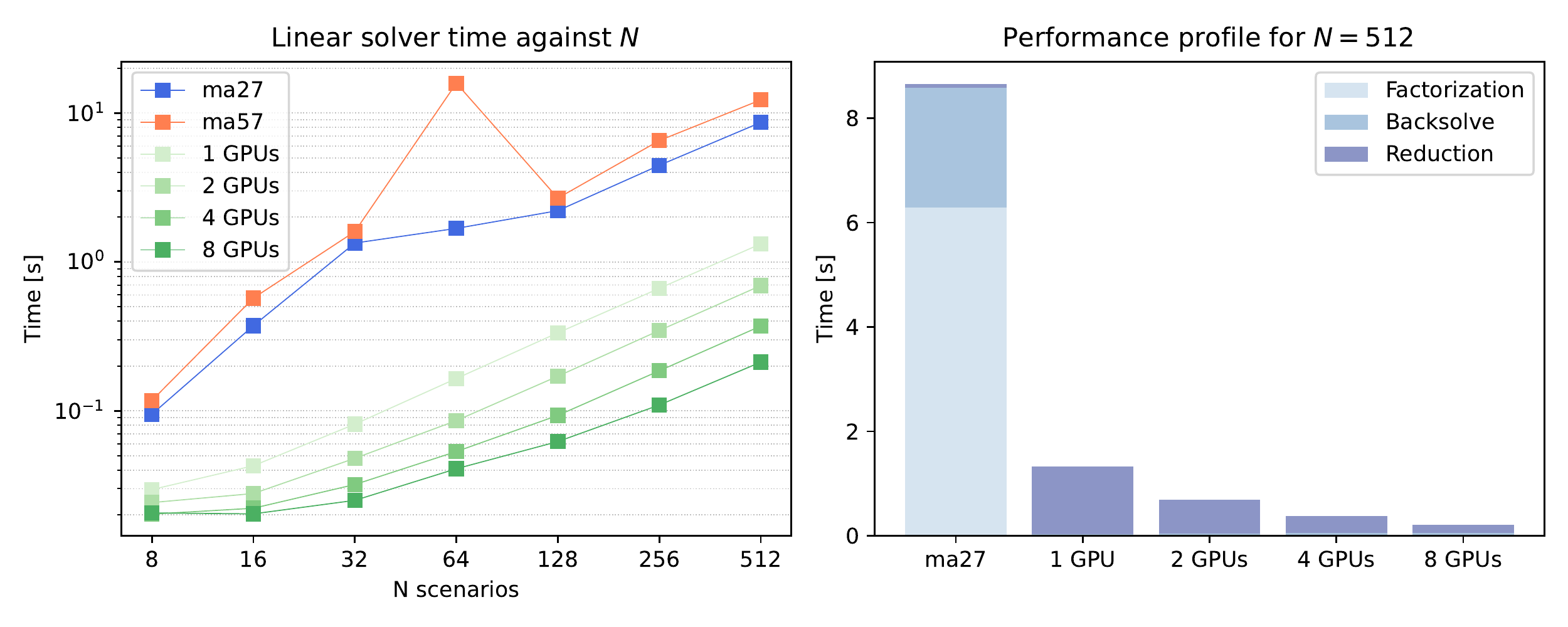}
  \caption{Time spent to solve the KKT system for {\tt case1354pegase}.}
  \label{fig:results:kkt}
\end{figure}

\subsubsection{Assessing the memory consumption}
We have observed in \S\ref{sec:algo:ad} that the total memory required
to store the duals is $\mathcal{O}((2n_x + n_d + n_b) \times M \times p)$,
with $M$ being the number of scenarios stored locally ($M = N$ on 1 GPU,
$M = N/2$ on 2 GPUs) and $p$ the number of tangents.
We display in Table~\ref{tab:results:memory} the memory taken by the automatic differentiation
backend and by the parallel KKT solver for {\tt case1354pegase} as we increase the number
of scenarios $N$.
We note that storing the duals is expensive in terms of memory,
with up to 10.9GB for $N=512$ on one GPU (as a reference, each NVIDIA A100 GPU
on Polaris has 40GB of memory available). By evaluating the model
on different processes with MPI, we can split the memory consumption
on the different GPUs we are using, leading to better use of the
resource at our disposal.

\begin{table}[!ht]
  \centering
  \begin{tabular}{lrr|rr}
    \toprule
  & \multicolumn{2}{c}{1 GPU}   & \multicolumn{2}{c}{2 GPUs}  \\
  \midrule
    $N$ & AD & KKT solver &AD & KKT solver \\
    \midrule
    8   & 171.1    & 92.3    & 85.5    & 48.1    \\
    16  & 342.2    & 181.5   & 171.1   & 93.1    \\
    32  & 684.3    & 360.0   & 342.2   & 183.2   \\
    64  & 1,368.7  & 716.8   & 684.3   & 363.2   \\
    128 & 2,737.3  & 1,430.5  & 1,368.7 & 723.4   \\
    256 & 5,474.7  & 2,858.0  & 2,737.3 & 1,443.6 \\
    512 & 10,949.3 & 5,712.8 & 5,474.7 & 2,884.1 \\
    \bottomrule
  \end{tabular}
  \caption{Memory consumption in MB}
  \label{tab:results:memory}
\end{table}

\subsection{Parallel solution of the block-structured OPF problem}
\label{sec:results:blockopf}
We analyze the parallel performance of our implementation
on block-structured OPF problems.

\subsubsection{Assessing the parallel performance w.r.t. the number of scenarios}
First, we are interested in the scaling of the parallel algorithm in relation
to the total number of scenarios $N$. We consider the {\tt case118} instance,
and increase the number of scenarios $N$ from 8 up to 2,048.
For each $N$, we solve the block-structured
OPF problem with MadNLP using our parallel KKT solver, and we compare with the performance
we obtained with HSL MA27.
The results are displayed in Figure~\ref{fig:results:opf_scenarios}.
We observe that the solver HSL MA27 is initially faster than our parallel KKT solver,
as the problem is too small to benefit from parallelism. However, as
soon as $N \geq 16$ the parallel KKT solver becomes competitive with HSL MA27.
The relative performance is improving as we increase the number of scenarios $N$:
for $N=512$, we get a 68x speed-up when using 8 GPUs, compared to the reference HSL MA27
(10.4s versus 712s).
Interestingly, using 2 nodes (=8 GPUs) does not lead to any speed-up compared to
a single node (=4 GPUs)
if $N \leq 256$; this setting is attractive only when the size of the problem becomes
sufficiently large ($N \geq 1024$) to compensate for the additional memory exchange.

\begin{figure}[!ht]
  \centering
  \includegraphics[width=12cm]{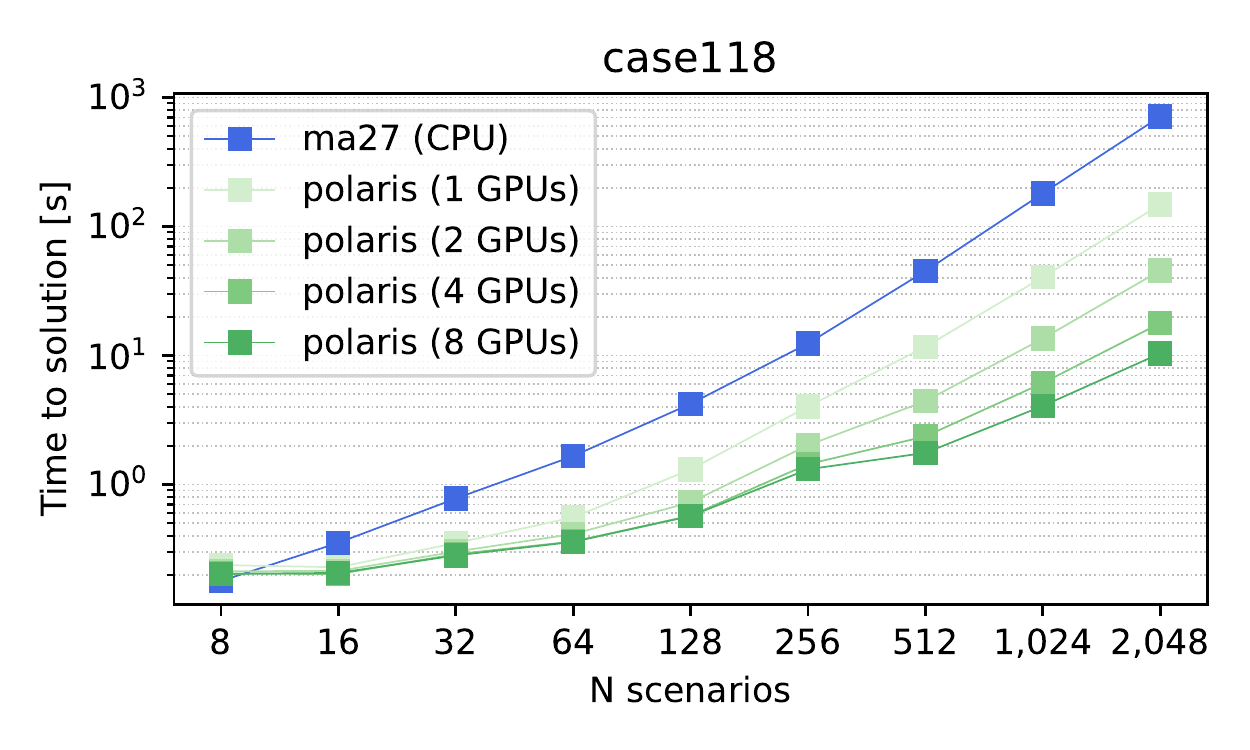}
  \caption{Time to solve the block-structured OPF problem {\tt case118}
  as a function of the number of scenarios $N$.}
  \label{fig:results:opf_scenarios}
\end{figure}

\subsubsection{Assessing the parallel performance w.r.t. the size of the problem}
Second, we increase the size of the problems.
We set a fixed number of scenarios $N = 8$, and look at the time to solution
for {\tt case1354pegase}, {\tt case2869pegase} and {\tt case9241pegase}.
We detail the respective dimension of each problem in Table~\ref{tab:results:8scens}.
We display the results in Figure~\ref{fig:results:opf_cases},
and give the detailed benchmark in Table~\ref{tab:results:bench}.
On the left (a), we display the total time required to find the solution
of the three instances as a function of the number of GPUs;
on the right (b), we show the performance profile
associated to {\tt case9241pegase}.
In (a), we observe that overall the parallel algorithm is faster than
the CPU implementation. The parallel algorithm
scales well as we increase the number of GPUs we are using,
the parallel algorithm being 35x faster than the reference when using
8 GPUs to solve {\tt case9241pegase}.
In (b), we detail the time spent in the different operations
for {\tt case9241pegase}: the time spent to factorize the Schur complement
with Lapack (using {\tt cusolve}) is constant as the size of the Schur complement
remains the same as we increase the number of GPUs. We observe that the time
spent in the AD decreases linearly with the number of GPUs exploited, but
the relative time spent in AD is negligible (less than 5\% of the total time).
Most of the time is spent in the parallel reduction,
as discussed earlier in \S\ref{sec:results:assessments:kkt}.

\begin{table}[!ht]
  \centering
  \begin{tabular}{lcrrrr}
    \toprule
    & $N$ & {\tt nvar} & {\tt ncon} & $\hat{K}_{uu}$ (mb) \\
    \midrule
    1354pegase & 8 & 20,095 & 53,520 & 2.1 \\
    2869pegase & 8 & 42,835 & 119,216 & 7.9 \\
    9241pegase & 8 & 139,177 & 404,640 & 63.7 \\
    \midrule
    1354pegase & 512 & 1,253,383 & 4,425,280 & 2.1 \\
    \bottomrule
  \end{tabular}
  \caption{Dimension of the instances we have used in our benchmark.\label{tab:results:8scens}}
\end{table}

\begin{figure}[!ht]
  \centering
  \includegraphics[width=.9\textwidth]{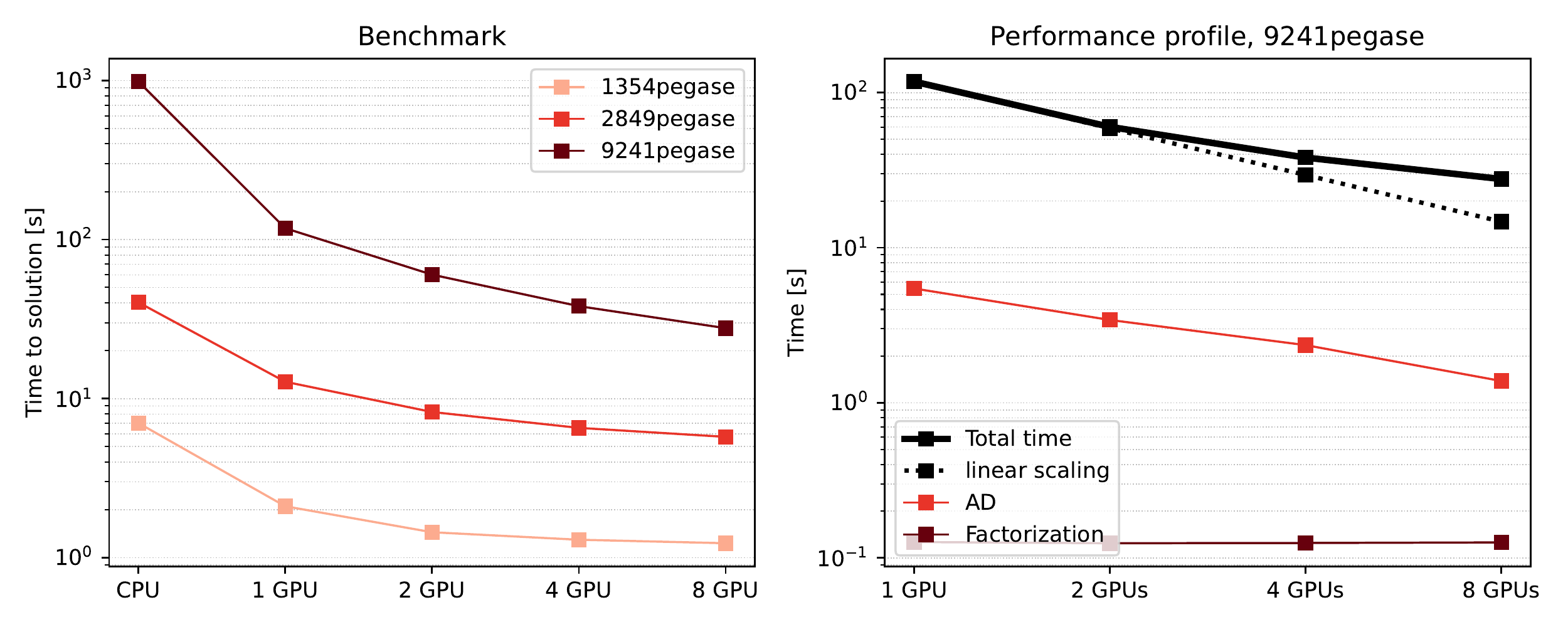}
  \caption{For a fixed number of scenarios $N=8$,
    (a) total time spent solving the block-OPF {\tt case1354pegase},
    {\tt case2869pegase} and {\tt case9241pegase} with MadNLP
    (b) performance profile for {\tt case9241pegase} with varying number of GPUs.
  }
  \label{fig:results:opf_cases}
\end{figure}

\begin{table}[!ht]
  \centering
  \resizebox{.95\textwidth}{!}{
    \begin{tabular}{lrrrr|rrrr|rrrr}
    \toprule
        &
    \multicolumn{4}{c}{\tt 1354pegase} &
    \multicolumn{4}{c}{\tt 2869pegase} &
    \multicolumn{4}{c}{\tt 9241pegase} \\
    \midrule
          &
  \#it & AD  & KKT  & {\bf Tot.}  &
  \#it & AD  & KKT  & {\bf Tot.}  &
  \#it & AD  & KKT  & {\bf Tot.}  \\
    \midrule
    CPU &
    44 & 2.6 & 4.2 & {\bf 7.0}  &
    77 & 11.9 & 27.4 & {\bf 40.3} &
    136 & 205.6 & 771.8 & {\bf 984.1} \\
    1 GPU &
    44 & 0.3 & 1.8 & {\bf 2.1}&
    93 & 1.1 & 11.7 & {\bf 12.8}&
    98 & 5.5 & 112.3 & {\bf 117.8}\\
    2 GPUs &
    44 & 0.3 & 1.1 & {\bf 1.4}&
    93 & 0.8 & 7.4 & {\bf 8.2}&
    98 & 3.4 & 56.8 & {\bf 60.2}\\
    4 GPUs &
    44 & 0.3 & 1.0 & {\bf 1.3} &
    93 & 0.8 & 5.7 & {\bf 6.5} &
    98 & 2.3 & 35.8 & {\bf 38.1}\\
    8 GPUs &
    44 & 0.2 & 1.0 & {\bf 1.2} &
    93 & 0.6 & 5.1 & {\bf 5.7} &
    98 & 1.4 & 26.4 & {\bf 27.7}\\
    \bottomrule
  \end{tabular}
  }
  \caption{Detailed results \label{tab:results:bench}}
\end{table}

\subsubsection{Assessing the parallel performance on a very large-scale instance}
We finish our numerical experiments by solving a very large-scale instance:
{\tt case1354pegase} with $N=512$ scenarios. The dimension of the resulting
optimization problem is displayed in Table~\ref{tab:results:8scens}: the problem
has more than 1 million variables, and 4 millions constraints.
We solve this instance on resp. 1 node, 2, 4 and 8 nodes (resp. 4, 8, 16 and 32 GPUs).
The results are displayed in Figure~\ref{fig:results:opf_large_scale}.
We observe that the scaling is almost perfect when we use 2 nodes (8 GPUs) instead of a single node
(4 GPUs) but we do not observe the same behavior when we increase the number of nodes to 4 and 8.
On that instance, the gain we get when using 8 nodes (32 GPUs) is marginal compared to when
using 4 nodes (16 GPUs): the solving time only decreases from 67s to 58s.
This corroborate our observations: it is better to pack all the computation on
a single node to use four A100 GPUs connected together via unified memory
(NVLINK has a transfer rate of 600GB/s). When
we have to use more than 2 nodes, the memory transfers are more involved as they
have to pass through the network of the supercomputer.

\begin{figure}[!ht]
  \centering
  \includegraphics[width=12cm]{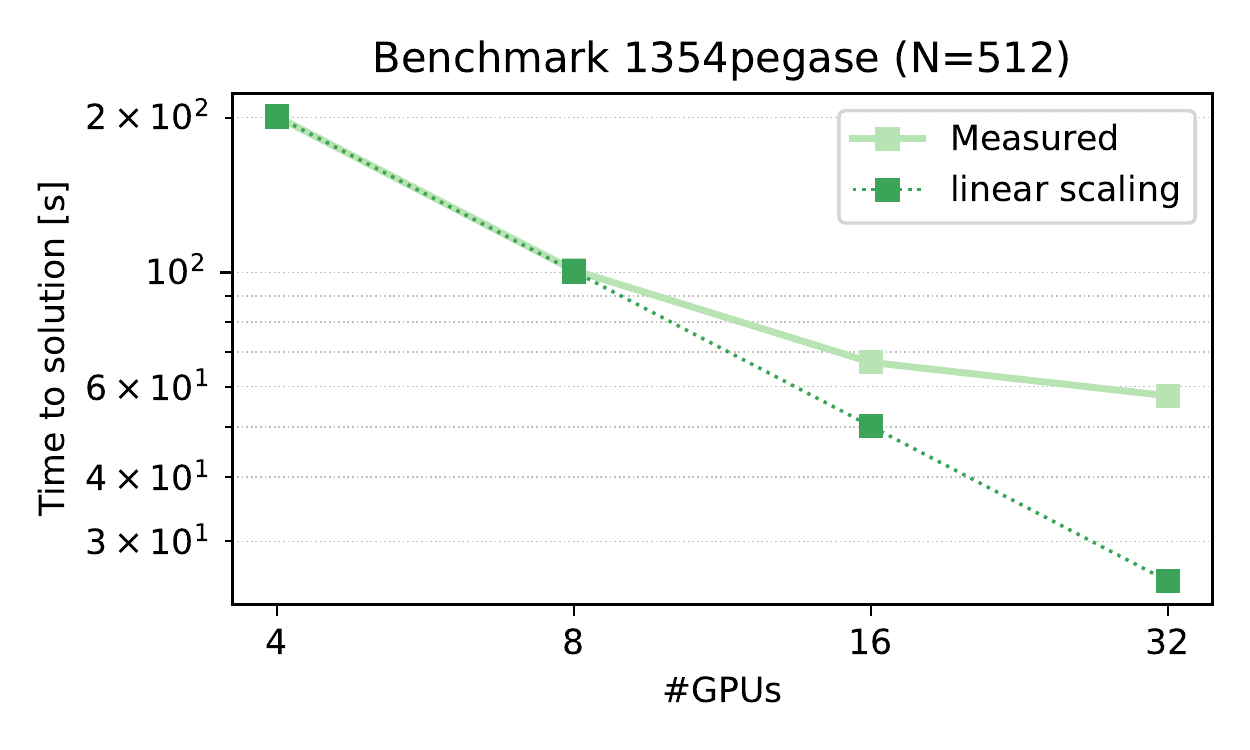}
  \caption{Solving {\tt case1354pegase} with $N=512$
    \label{fig:results:opf_large_scale}
  }
\end{figure}

\section{Conclusion}

We show promising results for leveraging massively parallel SIMD architectures
like GPUs for block-structured nonlinear programs. The parallelism is applied to
both the derivative evaluation and the solution of the KKT linear system. The
main operation in the KKT algorithm is the assembling of the Schur complement,
the factorization of the dense Schur complement being fast to carry on the GPU.

At all levels, the method benefits significantly from the massive parallelism,
achieving a speedup of around 40 for the derivatives compared to a
sequential CPU implementation. The speedup is very application dependent, not
least on the Hessian coloring and the problem's structure. The assembling
of the Schur complement is
bottlenecked by a distributed reduction operation bound by the
interconnect's latency and throughput between GPUs.
Current, so-called {\it super nodes} with multiple GPUs connected via fast networks like NVLINK greatly
accelerate this operation. Lastly, our method is limited by the memory capacity
of the GPU accelerators as it grows linearly with the number of problem blocks.
In the context of ACOPF we are confident that upcoming GPUs will provide enough
memory to solve a large number of scenarios in parallel, even for the largest
grid instances (e.g., Eastern Interconnection with 70,000 nodes).

With the upcoming release of the Aurora supercomputer, these SIMD architectures will
allow new science in regimes that were impossible with previous CPU architectures.

\section*{Acknowledgment}
This material was based upon work
supported by the U.S. Department of Energy, Office of Science,
Office of Advanced Scientific Computing Research (ASCR) under
Contract DE-AC02-06CH11347 and by NSF
through award CNS-1545046.
The authors gratefully acknowledge the funding support from the Applied Mathematics Program within the U.S. Department of Energy’s (DOE) Office of Advanced Scientific Computing Research (ASCR) as part of the project ExaSGD. This research used resources of the Argonne Leadership Computing Facility, which is a DOE Office of Science User Facility supported under Contract DE-AC02-06CH11357.

\vspace{0.1cm}
\begin{flushright}
  \scriptsize \framebox{\parbox{2.5in}{Government License: The
      submitted manuscript has been created by UChicago Argonne,
      LLC, Operator of Argonne National Laboratory (``Argonne").
      Argonne, a U.S. Department of Energy Office of Science
      laboratory, is operated under Contract
      No. DE-AC02-06CH11357.  The U.S. Government retains for
      itself, and others acting on its behalf, a paid-up
      nonexclusive, irrevocable worldwide license in said
      article to reproduce, prepare derivative works, distribute
      copies to the public, and perform publicly and display
      publicly, by or on behalf of the Government. The Department of Energy will provide public access to these results of federally sponsored research in accordance with the DOE Public Access Plan. http://energy.gov/downloads/doe-public-access-plan. }}
  \normalsize
\end{flushright}


\begin{thebibliography}{}

\bibitem[Amestoy et~al., 2000]{amestoy2000multifrontal}
Amestoy, P.~R., Duff, I.~S., and L'excellent, J.-Y. (2000).
\newblock Multifrontal parallel distributed symmetric and unsymmetric solvers.
\newblock {\em Computer methods in applied mechanics and engineering},
  184(2-4):501--520.

\bibitem[Birge and Qi, 1988]{birge1988computing}
Birge, J.~R. and Qi, L. (1988).
\newblock Computing block-angular {Karmarkar} projections with applications to
  stochastic programming.
\newblock {\em Management science}, 34(12):1472--1479.

\bibitem[Bradbury et~al., 2018]{jax2018github}
Bradbury, J., Frostig, R., Hawkins, P., Johnson, M.~J., Leary, C., Maclaurin,
  D., Necula, G., Paszke, A., Vander{P}las, J., Wanderman-{M}ilne, S., and
  Zhang, Q. (2018).
\newblock {JAX}: composable transformations of {P}ython+{N}um{P}y programs.

\bibitem[B{\"u}cker et~al., 2001]{bucker2001bringing}
B{\"u}cker, H.~M., Lang, B., an~Mey, D., and Bischof, C.~H. (2001).
\newblock Bringing together automatic differentiation and {OpenMP}.
\newblock In {\em Proceedings of the 15th international conference on
  Supercomputing}, pages 246--251.

\bibitem[Bussieck and Meeraus, 2004]{bussieck2004general}
Bussieck, M.~R. and Meeraus, A. (2004).
\newblock General algebraic modeling system ({GAMS}).
\newblock In {\em Modeling languages in mathematical optimization}, pages
  137--157. Springer.

\bibitem[Byrne et~al., 2021]{byrne2021mpi}
Byrne, S., Wilcox, L.~C., and Churavy, V. (2021).
\newblock {MPI. jl}: Julia bindings for the {M}essage {P}assing {I}nterface.
\newblock In {\em Proceedings of the JuliaCon Conferences}, volume~1, page~68.

\bibitem[Capitanescu et~al., 2011]{capitanescu2011state}
Capitanescu, F., Ramos, J.~M., Panciatici, P., Kirschen, D., Marcolini, A.~M.,
  Platbrood, L., and Wehenkel, L. (2011).
\newblock State-of-the-art, challenges, and future trends in security
  constrained optimal power flow.
\newblock {\em Electric power systems research}, 81(8):1731--1741.

\bibitem[Chiang et~al., 2014]{chiang2014structured}
Chiang, N., Petra, C.~G., and Zavala, V.~M. (2014).
\newblock Structured nonconvex optimization of large-scale energy systems using
  {PIPS-NLP}.
\newblock In {\em 2014 Power Systems Computation Conference}, pages 1--7. IEEE.

\bibitem[Choi and Goldfarb, 1993]{choi1993exploiting}
Choi, I.~C. and Goldfarb, D. (1993).
\newblock Exploiting special structure in a primal—dual path-following
  algorithm.
\newblock {\em Mathematical Programming}, 58(1):33--52.

\bibitem[Colombo et~al., 2009]{colombo2009structure}
Colombo, M., Grothey, A., Hogg, J., Woodsend, K., and Gondzio, J. (2009).
\newblock A structure-conveying modelling language for mathematical and
  stochastic programming.
\newblock {\em Mathematical Programming Computation}, 1(4):223--247.

\bibitem[DeMiguel and Nogales, 2008]{demiguel2008decomposition}
DeMiguel, V. and Nogales, F.~J. (2008).
\newblock On decomposition methods for a class of partially separable nonlinear
  programs.
\newblock {\em Mathematics of Operations Research}, 33(1):119--139.

\bibitem[Duff, 2004]{duff2004ma57}
Duff, I.~S. (2004).
\newblock {MA57}---a code for the solution of sparse symmetric definite and
  indefinite systems.
\newblock {\em ACM Transactions on Mathematical Software (TOMS)},
  30(2):118--144.

\bibitem[Duff and Van Der~Vorst, 1999]{duff1999developments}
Duff, I.~S. and Van Der~Vorst, H.~A. (1999).
\newblock Developments and trends in the parallel solution of linear systems.
\newblock {\em Parallel Computing}, 25(13-14):1931--1970.

\bibitem[Dunning et~al., 2017]{dunning2017jump}
Dunning, I., Huchette, J., and Lubin, M. (2017).
\newblock {JuMP}: A modeling language for mathematical optimization.
\newblock {\em SIAM review}, 59(2):295--320.

\bibitem[Fourer et~al., 1990]{fourer1990modeling}
Fourer, R., Gay, D.~M., and Kernighan, B.~W. (1990).
\newblock A modeling language for mathematical programming.
\newblock {\em Management Science}, 36(5):519--554.

\bibitem[Gondzio and Grothey, 2009]{gondzio2009exploiting}
Gondzio, J. and Grothey, A. (2009).
\newblock Exploiting structure in parallel implementation of interior point
  methods for optimization.
\newblock {\em Computational Management Science}, 6(2):135--160.

\bibitem[Gondzio and Sarkissian, 2003]{gondzio2003parallel}
Gondzio, J. and Sarkissian, R. (2003).
\newblock Parallel interior-point solver for structured linear programs.
\newblock {\em Mathematical Programming}, 96(3):561--584.

\bibitem[Griewank and Walther, 2008]{griewank2008evaluating}
Griewank, A. and Walther, A. (2008).
\newblock {\em Evaluating derivatives: principles and techniques of algorithmic
  differentiation}.
\newblock SIAM.

\bibitem[Hovland and Bischof, 1998]{hovland1998automatic}
Hovland, P. and Bischof, C. (1998).
\newblock Automatic differentiation for message-passing parallel programs.
\newblock In {\em Proceedings of the First Merged International Parallel
  Processing Symposium and Symposium on Parallel and Distributed Processing},
  pages 98--104. IEEE.

\bibitem[Hovland, 1997]{hovland1997automatic}
Hovland, P.~D. (1997).
\newblock {\em Automatic differentiation of parallel programs}.
\newblock University of Illinois at Urbana-Champaign.

\bibitem[Huchette et~al., 2014]{huchette2014parallel}
Huchette, J., Lubin, M., and Petra, C. (2014).
\newblock Parallel algebraic modeling for stochastic optimization.
\newblock In {\em 2014 First Workshop for High Performance Technical Computing
  in Dynamic Languages}, pages 29--35. IEEE.

\bibitem[Jessup et~al., 1994]{jessup1994parallel}
Jessup, E.~R., Yang, D., and Zenios, S.~A. (1994).
\newblock Parallel factorization of structured matrices arising in stochastic
  programming.
\newblock {\em SIAM journal on Optimization}, 4(4):833--846.

\bibitem[Kardo{\v{s}} et~al., 2019]{kardovs2019two}
Kardo{\v{s}}, J., Kourounis, D., and Schenk, O. (2019).
\newblock Two-level parallel augmented {S}chur complement interior-point
  algorithms for the solution of security constrained optimal power flow
  problems.
\newblock {\em IEEE Transactions on power systems}, 35(2):1340--1350.

\bibitem[Kardo{\v{s}} et~al., 2020]{kardovs2020structure}
Kardo{\v{s}}, J., Kourounis, D., and Schenk, O. (2020).
\newblock Structure-exploiting interior point methods.
\newblock In {\em Parallel Algorithms in Computational Science and
  Engineering}, pages 63--93. Springer.

\bibitem[Kardo{\v{s}} et~al., 2022]{kardovs2022beltistos}
Kardo{\v{s}}, J., Kourounis, D., Schenk, O., and Zimmerman, R. (2022).
\newblock Beltistos: A robust interior point method for large-scale optimal
  power flow problems.
\newblock {\em Electric Power Systems Research}, 212:108613.

\bibitem[Lubin et~al., 2011]{lubin2011scalable}
Lubin, M., Petra, C.~G., Anitescu, M., and Zavala, V. (2011).
\newblock Scalable stochastic optimization of complex energy systems.
\newblock In {\em SC'11: Proceedings of 2011 International Conference for High
  Performance Computing, Networking, Storage and Analysis}, pages 1--10. IEEE.

\bibitem[Moses et~al., 2021]{moses2021reverse}
Moses, W.~S., Churavy, V., Paehler, L., H{\"u}ckelheim, J., Narayanan, S.
  H.~K., Schanen, M., and Doerfert, J. (2021).
\newblock Reverse-mode automatic differentiation and optimization of {GPU}
  kernels via {Enzyme}.
\newblock In {\em Proceedings of the International Conference for High
  Performance Computing, Networking, Storage and Analysis}, pages 1--16.

\bibitem[Nash and Sofer, 1989]{nash1989block}
Nash, S.~G. and Sofer, A. (1989).
\newblock Block truncated-{N}ewton methods for parallel optimization.
\newblock {\em Mathematical Programming}, 45(1):529--546.

\bibitem[Nash and Sofer, 1991]{nash1991general}
Nash, S.~G. and Sofer, A. (1991).
\newblock A general-purpose parallel algorithm for unconstrained optimization.
\newblock {\em SIAM Journal on Optimization}, 1(4):530--547.

\bibitem[Nocedal and Wright, 2006]{nocedal2006numerical}
Nocedal, J. and Wright, S.~J. (2006).
\newblock {\em Numerical optimization}.
\newblock Springer series in operations research. Springer, New York, 2nd
  edition.

\bibitem[Pacaud et~al., 2022]{pacaud2022condensed}
Pacaud, F., Shin, S., Schanen, M., Maldonado, D.~A., and Anitescu, M. (2022).
\newblock Accelerating condensed interior-point methods on {SIMD/GPU}
  architectures.
\newblock {\em arXiv preprint arXiv:2203.11875}.

\bibitem[Paszke et~al., 2019]{paszke2019pytorch}
Paszke, A., Gross, S., Massa, F., Lerer, A., Bradbury, J., Chanan, G., Killeen,
  T., Lin, Z., Gimelshein, N., Antiga, L., et~al. (2019).
\newblock Pytorch: An imperative style, high-performance deep learning library.
\newblock {\em Advances in neural information processing systems}, 32.

\bibitem[Petra et~al., 2014]{petra2014augmented}
Petra, C.~G., Schenk, O., Lubin, M., and G{\"a}rtner, K. (2014).
\newblock An augmented incomplete factorization approach for computing the
  {S}chur complement in stochastic optimization.
\newblock {\em SIAM Journal on Scientific Computing}, 36(2):C139--C162.

\bibitem[Rodriguez et~al., 2021]{rodriguez2021scalable}
Rodriguez, J.~S., Parker, R., Laird, C.~D., Nicholson, B., Siirola, J.~D., and
  Bynum, M. (2021).
\newblock Scalable parallel nonlinear optimization with {PyNumero} and
  {Parapint}.
\newblock {\em Optimization Online}.

\bibitem[Ruszczynski, 1993]{ruszczynski1993interior}
Ruszczynski, A. (1993).
\newblock Interior point methods in stochastic programming.
\newblock {\em Working paper}.

\bibitem[Saad, 1980]{saad1980rates}
Saad, Y. (1980).
\newblock On the rates of convergence of the {Lanczos} and the block-{Lanczos}
  methods.
\newblock {\em SIAM Journal on Numerical Analysis}, 17(5):687--706.

\bibitem[Schanen et~al., 2018]{schanen2018toward}
Schanen, M., Gilbert, F., Petra, C.~G., and Anitescu, M. (2018).
\newblock Toward multiperiod ac-based contingency constrained optimal power
  flow at large scale.
\newblock In {\em 2018 Power Systems Computation Conference (PSCC)}, pages
  1--7. IEEE.

\bibitem[Schenk and G{\"a}rtner, 2004]{schenk2004solving}
Schenk, O. and G{\"a}rtner, K. (2004).
\newblock Solving unsymmetric sparse systems of linear equations with
  {PARDISO}.
\newblock {\em Future Generation Computer Systems}, 20(3):475--487.

\bibitem[Schnabel, 1985]{schnabel1985parallel}
Schnabel, R.~B. (1985).
\newblock Parallel computing in optimization.
\newblock In {\em Computational Mathematical Programming}, pages 357--381.
  Springer.

\bibitem[Schnabel, 1995]{schnabel1995view}
Schnabel, R.~B. (1995).
\newblock A view of the limitations, opportunities, and challenges in parallel
  nonlinear optimization.
\newblock {\em Parallel computing}, 21(6):875--905.

\bibitem[Shin et~al., 2021]{shin2021graph}
Shin, S., Coffrin, C., Sundar, K., and Zavala, V.~M. (2021).
\newblock Graph-based modeling and decomposition of energy infrastructures.
\newblock {\em IFAC-PapersOnLine}, 54(3):693--698.

\bibitem[Tasseff et~al., 2019]{tasseff2019exploring}
Tasseff, B., Coffrin, C., W{\"a}chter, A., and Laird, C. (2019).
\newblock Exploring benefits of linear solver parallelism on modern nonlinear
  optimization applications.
\newblock {\em arXiv preprint arXiv:1909.08104}.

\bibitem[Watson et~al., 2012]{watson2012pysp}
Watson, J.-P., Woodruff, D.~L., and Hart, W.~E. (2012).
\newblock {PySP}: modeling and solving stochastic programs in python.
\newblock {\em Mathematical Programming Computation}, 4(2):109--149.

\bibitem[Word et~al., 2014]{word2014efficient}
Word, D.~P., Kang, J., Akesson, J., and Laird, C.~D. (2014).
\newblock Efficient parallel solution of large-scale nonlinear dynamic
  optimization problems.
\newblock {\em Computational Optimization and Applications}, 59(3):667--688.

\bibitem[Zavala et~al., 2008]{zavala2008interior}
Zavala, V.~M., Laird, C.~D., and Biegler, L.~T. (2008).
\newblock Interior-point decomposition approaches for parallel solution of
  large-scale nonlinear parameter estimation problems.
\newblock {\em Chemical Engineering Science}, 63(19):4834--4845.

\bibitem[Zhu et~al., 2009]{zhu2009exploiting}
Zhu, Y., Word, D., Siirola, J., and Laird, C.~D. (2009).
\newblock Exploiting modern computing architectures for efficient large-scale
  nonlinear programming.
\newblock In {\em Computer Aided Chemical Engineering}, volume~27, pages
  783--788. Elsevier.

\bibitem[Zimmerman et~al., 2010]{zimmerman2010matpower}
Zimmerman, R.~D., Murillo-S{\'a}nchez, C.~E., and Thomas, R.~J. (2010).
\newblock {MATPOWER}: {Steady-state operations, planning, and analysis tools
  for power systems research and education}.
\newblock {\em IEEE Transactions on Power Systems}, 26(1):12--19.

\end{thebibliography}
\end{document}